\documentclass{amsart}[11pt]
\usepackage{enumerate}
\usepackage{a4wide}
\usepackage[english]{babel}

\usepackage[mathscr]{euscript}
\usepackage{epsfig}
\usepackage{latexsym}

\usepackage{amssymb}
\usepackage{amsthm}
\usepackage{amsmath}
\usepackage{amsxtra}

.tfm
.tfm

\theoremstyle{plain}
\newtheorem{prop}{Proposition}[section]
\newtheorem*{thmA}{Theorem A}
\newtheorem*{thmB}{Theorem B}
\newtheorem*{prop*}{Proposition}

\newtheorem*{thm*}{Theorem}
\newtheorem{coro}[prop]{Corollary}

\newtheorem{lemma}[prop]{Lemma}

\theoremstyle{definition}

\theoremstyle{remark}

\newtheorem*{remark}{Remark}

\numberwithin{table}{section}

\DeclareMathOperator{\Frob}{Frob}

\DeclareMathOperator{\Gal}{Gal}

\DeclareMathOperator{\coho}{H}

\DeclareMathOperator{\sha}{III}

\newcommand{\F}{\mathbb F}

\newcommand{\HH}{\mathcal H}
\newcommand{\LL}{\mathcal N}

\newcommand{\Om}{{\mathscr{O}}}

\newcommand{\nequiv}{\not\equiv}

\def\FF{\mathbb F}

\def\<#1>{{\left\langle{#1}\right\rangle}}

\def\Z{{\mathbb Z}}             
\def\Q{{\mathbb Q}}             
\def\C{{\mathbb C}}             

\def\set#1{{\left\{{\def\st{\;:\;}#1}\right\}}}

\def\O{{\mathcal{O}}}           
\def\Ol(#1){{\mathop{\O_l}(\id{#1})}}                 
\def\Or(#1){{\mathop{\O_r}(\id{#1})}}                 
\def\Orp(#1){{\mathop{\O_r}(\idp{#1})}}               
\def\Otern(#1){{\mathop{\O^0_r}(\id{a})}}    

\def\id#1{{\mathfrak{#1}}}      
\def\idp#1{{\id{#1}_p}}         

\def\T_#1(#2){{\mathop{\mathscr T}\nolimits_{#1}(\id{#2})}}
\def\TO(#1)_#2(#3){{\mathop{\mathscr T}\nolimits^{#1}_{#2}(\id{#3})}}

\def\HeckeRing{\mathbb{T}}

\def\A_#1(#2){{\mathop{\mathscr A}\nolimits_{#1}(\id{#2})}}
\def\Ax_#1(#2){{\mathop{\widetilde{\mathscr A}}\nolimits_{#1}(\id{#2})}}

\def\Ad{Ad^0\bar{\rho}}
\def\Adr{\widetilde{Ad}^0\bar{\rho}}
\def\Add{Ad^0{\rho}}

\def\rrho{\overline{\rho}}
\def\rrhon{\overline{\rho_n}}
\def\trho{\tilde{\rho}}
\def\T{\mathcal{T}}
\DeclareMathOperator{\GL}{GL}
\DeclareMathOperator{\SL}{SL}
\DeclareMathOperator{\PSL}{PSL}
\DeclareMathOperator{\PGL}{PGL}
\DeclareMathOperator{\Img}{Im}
\DeclareMathOperator{\Ind}{Ind}
\DeclareMathOperator{\Ker}{Ker}

\usepackage[latin1]{inputenc}
\usepackage[all]{xy}
\usepackage{color}
\usepackage[active]{srcltx}
\usepackage{url}
\begin{document}
\title{Congruences between modular forms modulo prime powers}

\author{Maximiliano Camporino}
\address{Departamento de Matem\'atica, Facultad de Ciencias Exactas y Naturales, Universidad de Buenos Aires} 
\email{maxicampo@gmail.com}
\thanks{MC was partially supported by a CONICET doctoral fellowship}

\author{Ariel Pacetti}
\address{Departamento de Matem\'atica, Facultad de Ciencias Exactas y Naturales, Universidad de Buenos Aires and IMAS, CONICET, Argentina}
\email{apacetti@dm.uba.ar}
\thanks{AP was partially supported by CONICET PIP 2010-2012 GI and FonCyT BID-PICT 2010-0681.}

\begin{abstract}
  Given a prime $p \ge 5$ and an abstract odd representation $\rho_n$
  with coefficients modulo $p^n$ (for some $n \ge 1$) and big image,
  we prove the existence of a lift of $\rho_n$ to characteristic $0$
  whenever local lifts exist (under some technical
  conditions). Moreover, we can chose the inertial type of our lift at
  all primes but finitely many (where the lift is of Steinberg type).
 
  We apply this result to the realm of modular forms, proving a level
  lowering theorem modulo prime powers and providing examples of level
  raising.  In particular, our method shows that given a modular eigenform
  $f$ without Complex Multiplication or inner twists, for all primes $p$ but finitely
  many, and for all positive integers $n$, there exists another
  eigenform $g\neq f$, which is congruent to $f$ modulo $p^n$.
 
\end{abstract}

\maketitle

\section{Introduction}

The aim of the present article is to deal with congruences between
modular forms (and more generally, abstract representations) modulo
prime powers. The main strategy of the paper is to adapt the arguments
of \cite{ravi2} and \cite{ravi3} to this new setting, which is harder
due to semisimplification problems. Let $\F$ be a finite field of
residual characteristic $p$, and $\rho_n:G_\Q \to \GL_2(W(\F)/p^n)$ be
a continuous representation. We denote by $\overline{\rho_n}$ its
reduction modulo $p$.



%
%
%


If $T$ is a finite set of primes, we denote by $G_T$ the Galois group
of $\Gal(\Q_T/\Q)$, where $\Q_T$ is the maximal extension of $\Q$
unramified outside $T$, and by $G_\Q$ we will denote the whole Galois group $\Gal(\overline{\Q}/\Q)$. One of the main results of this work is the following.

\begin{thmA}
  Let $\F$ be a finite field of characteristic $p>5$.  Consider
  $\rho_n: G_\Q \to \GL_2(W(\F)/p^n)$ a continuous representation
  ramified at a finite set of primes $S$ satisfying the following properties:
  \begin{itemize}
  \item The image is big, i.e. $\SL_2(\F) \subseteq \Img(\rrhon)$.
  \item $\rho_n$ is odd.
  \item The restriction $\overline{\rho_n}|_{G_p}$ is not twist equivalent
    to the trivial representation nor the indecomposable unramified
    representation given by $\left(\begin{smallmatrix} 1 &*\\ 0 &
        1\end{smallmatrix}\right)$.


  \end{itemize}
Let $P$ be a finite set of primes containing $S$, and for every
  $\ell\in P$, $\ell \neq p$, fix a deformation 
$\rho_\ell:G_\ell \to W(\F)$ of $\rho_n|_{G_\ell}$.
At the prime $p$, let $\rho_p$ be
a deformation of $\rho_n|_{G_p}$ which is ordinary or crystalline with
Hodge-Tate weights $\{0,k\}$, with $2 \le k \le p-1$.

Then there is a finite set $Q$ of auxiliary primes $q \nequiv \pm1 \pmod{p}$ and a modular representation
\[
\rho: G_{P\cup Q} \longrightarrow \GL_2(W(\F)),
\]
such that:
\begin{itemize}
\item the reduction modulo $p^n$ of $\rho$ is $\rho_n$,
\item $\rho|_{I_\ell} \simeq \rho_\ell|_{I_\ell}$ for every $\ell \in P$,
\item $\rho|_{G_q}$ is a ramified representation of Steinberg type for every $q \in Q$.
\end{itemize}
 \end{thmA}

 This result, contrary to the results of Ramakrishna, is only about odd
 representations (and hence modular by Serre's conjectures). In the
 even case, the exact same ideas plus some extra hypothesis (as in
 \cite{ravi2}) give a result for any abstract representation with big image.


 \begin{remark}
   Theorem A is in the same spirit as Theorem $3.2.2$ of \cite{BD},
   where they only consider residual representations, and allow the
   coefficient field to grow. The advantage of their method is that it
   does not require to add extra ramification (so $Q = \emptyset$),
   but this phenomena only works while working modulo a prime. For
   example, the elliptic curve $329a1$ is unramified at $7$ modulo
   $9$, but there are no newforms of level $47$ congruent to it modulo
   $9$ (see \cite{dummigan}).
 \end{remark}

 \begin{coro}[Lowering the level] Let $\rho_n:G_\Q\to
   \GL_2(\Om_f/\id{p}^n)$ be a representation attached to a modular
   form $f \in S_k(\Gamma_0(M),\varepsilon)$, and suppose that:
   \begin{itemize}
   \item $p \ge 5$.
   \item $2 \le k \le p-1$.
   \item $\SL_2(\Om_f/\id{p}) \subseteq \Img(\overline{\rho_n})$.
   \item $p$ does not ramify in $\Om_f$.
   \end{itemize}

   If $\ell \mid M$ is such that $\rho_n$ is unramified at $\ell$,
   then the Hecke map factors through the $\ell$-old quotient
   $\HeckeRing_k^{\ell\text{-old}}(M,\ell)$.
\label{coro:loweringthelevel}
 \end{coro}

 \begin{proof} The proof just consists on combining the result for
   primes $\ell \nequiv 1 \pmod p$ (which was proved in
   \cite{dummigan}, Theorem 1), with Theorem A that allows us to move the
   ramified primes to a situation where we get more control on the
   extra Steinberg ramification. Specifically, if $\ell \equiv 1 \pmod
   p$, then by Theorem A, we can find a form $g$ with the same
   ramification as $f$, but without $\ell$ in the level at the cost of
   adding many Steinberg primes $q \nequiv 1 \pmod p$. But these extra
   primes in the level of the form $g$ satisfy the hypotheses of Dummigan's Theorem,
   so we can remove them as well.
 \end{proof}

 Let $f\in S_k(\Gamma_0(N),\epsilon)$ ($k \ge 2$) be a newform, with
 coefficient field $K$.  Let $\Om_K$ denote the ring of integers of
 $K$. If $p$ is a prime number, let $\id{p}$ denote a prime ideal in
 $\Om_K$ dividing $p$ with ramification index $1$ and $\Om_\id{p}$ the
 completion at $\id{p}$. Finally let
 $\rho_{f,p}:\Gal(\overline{\Q}/\Q) \to \GL_2(K_{\id{p}})$ denote its
 associated $p$-adic Galois representation. If $n$ is a positive
 integer, let
\[
\rho_n:G_\Q \to \GL_2(\Om_{\id{p}}/{\id{p}^n})
\] 
be its reduction. Applying Theorem A to this representation, we are
able to derive the other main result of this paper.
 
\begin{thmB}
  In the above hypothesis, let $n$ be a positive integer and $p>k$ be
  a prime such that:
  \begin{itemize}
  \item $p\nmid N$ or $f$ is ordinary at $p$,
  \item $\SL_2(\Om_{\id{p}}/\id{p}) \subseteq \Img(\overline{\rho_{f,p}})$,
  \end{itemize}
  Let $N'$ denote the conductor of $\rho_n$. If $N''=\prod_{p \mid
    N'}p^{v_p(N)}$, i.e. $N''$ is $N$ divided by its prime-to-$N'$
  part, then there exist an integer $r$, a set $\{q_1,\ldots,q_r\}$ of
  auxiliary primes prime to $N$ satisfying $q_i \not \equiv 1 \pmod p$
  and a newform $g$, different from $f$, of weight $k$ and level
  $N''q_1\ldots q_r$ such that $f$ and $g$ are congruent modulo
  $p^n$. Furthermore, the form $g$ can be chosen with the same
  restriction to inertia as that of $f$ at the primes dividing $N'$.
\end{thmB}

\begin{coro}
  Let $f\in S_k(\Gamma_0(N),\epsilon)$, $k \ge 2$ be a newform which
  has no complex multiplication or inner twists. Then for all but
  finitely many prime numbers $p$, and for all positive integers $n$,
  there exists a weight $k$ newform $g$ (depending on $p$ and $n$)
  different from $f$, which is congruent to $f$ modulo $p^n$.
\label{coro:congruenceproblem}
\end{coro}

\begin{proof}
  Since our form does not have complex multiplication or inner twists, by Ribet's
  result (\cite{Ribet}, Theorem 3.1) the image is big modulo $p$ for
  all but finitely many primes $p$. We avoid the primes without big
  image as well as those smaller than the weight. We also discard the
  primes $p$ that ramify in the field of coefficients of $f$ and 
  the ones in the level (or the non-ordinary
  ones), and we are in the hypothesis of the previous Theorem.
\end{proof}

The proof of Theorem A follows the ideas of \cite{ravi3}. This means
that it is divided into two parts. On the one hand we need to add
auxiliary primes that allow us to convert the problem of lifting a
global representation into the one of lifting many local ones. On the
other hand, we need to solve the local problems.  Following the
logical structure of \cite{ravi3}, we deal with the local
considerations first.

In this case, we essentially have to prove Proposition 1.6 of
\cite{ravi3} in our setting. For every prime $\ell \in P$ we need to
find a set $C_\ell$ of deformations of $\rho_n|_{G_\ell}$ to $W(\F)$
containing $\rho_\ell$ and a subspace $N_\ell \subseteq
\coho^1(G_\ell, \Ad)$ of certain dimension such that its elements
preserve the reductions of $N_\ell$, i.e. such that whenever $\rho_m$
is the reduction of some $\tilde{\rho} \in C_\ell$ modulo $p^m$ and $u
\in N_\ell$ then $(1+p^{m-1}u)\rho_m$ is the reduction of some other
$\tilde{\rho}'\in C_\ell$.  In order to get the full statement of our
Theorem A we also need all the deformations in $C_\ell$ to be
isomorphic when restricted to $I_\ell$.

Once we picked these local deformations classes, 
we need to construct two auxiliary sets of primes, $Q_1$ and
$Q_2$ (these are Ramakrishna's $Q$ and $T$) together with their respective 
sets $C_q$ and subspaces $N_q$ as for the primes in $P$, that satisfy the following
conditions:

\begin{itemize}
\item The set $Q_1$ morally has two main properties (see Fact 16
  \cite{ravi3}): it kills the global obstructions, i.e. is such that
  $\sha^1_{S\cup Q_1}((\Ad)^*)=0$ and therefore $\sha^2_{S\cup Q_1}(\Ad)=0$, and the inflation map
\[
\coho^2(G_S,\Ad)\to \coho^2(G_{S\cup Q_1},\Ad),
\]
is an isomorphism.
\item The set $Q_2$ gives an isomorphism
\[
\coho^1(G_{S\cup Q_1 \cup Q_2},\Ad) \to \bigoplus_{\ell \in S \cup Q_1 \cup Q_2} \coho^1(G_\ell,\Ad)/N_{\ell}.
\]

without adding global obstructions, i.e. $\sha^2_{S\cup Q_1 \cup Q_2} = 0$.
\end{itemize}
These auxilliary primes are essentially the same as in \cite{ravi3}, we use 
the same sets $C_q$ and subspace $N_q$. We only need to have a little 
extra care when proving that $\rho_n|_{G_q}$ is the reduction of some $\tilde{\rho} \in C_q$
for every $q \in Q_1 \cup Q_2$.

Once we have solved the local problems and found the auxiliary
primes, the inductive method starts to work. The key observation here
is that this inductive step only depends on hypotheses about the
reduction mod $p$ of our representation, which tells us that no matter
at which power of $p$ we start lifting, it will work perfectly.

The inductive argument works as follows: in virtue of the isomorphisms
between local and global second cohomology groups, a global
deformation to $W(\F)/p^m$ lifts to $W(\F)/p^{m+1}$ if and only if its
restrictions to the primes of $P\cup Q_1 \cup Q_2$ lift to
$W(\F)/p^{m+1}$. For $m = n$ the local condition is automatic so there
exists a lift $\rho_{n+1}$ of $\rho_n$ to $W(\F)/p^{n+1}$.
The problem is that $\rho_{n+1}$ may not lift again, as it can be
locally obstructed. In order to remove these local obstructions we use
the fact that any local deformation for primes $\ell \in P\cup Q_1
\cup Q_2$ can be modified by some element not in $N_\ell$ in order to
be a reduction of some element of $C_\ell$ and therefore
unobstructed. We will often refer to this as \emph{adjusting a local
deformation}. As we have an isomorphism between the global first
cohomology group and the local first cohomology groups modulo
$N_\ell$, we can find an element $u \in \coho^1(G_\Q, \Ad)$ that
adjusts $\rho_{n+1}$ locally for every prime in $P \cup Q_1 \cup Q_2$
making $(1+p^nu)\rho_{n+1}$ an unobstructed lift of $\rho_n$.  From
here we can repeat the process of lifting and adjusting indefinitely,
finally getting a lift to $W(\F)$.

Finally, to get Theorem A we need to prove modularity for the 
constructed representation, this follows from the appropriate modularity lifting
theorem, using the conditions we chose for the representation at $p$. 

Theorem B is an immediate consequence of Theorem A. The fact $f
\neq g$ will follow from the fact that both forms have different
levels, as the auxilliary primes involved necessarily ramify. If there are
no auxilliary primes, we add a ramified prime into the set $P$.

\medskip

\noindent{\bf Notations and conventions:} throughout this work we will
denote by $G_\Q$ the Galois group $\Gal(\overline{\Q}/\Q)$. If $\ell$ is a
prime, we denote by $G_{\ell}$ a decomposition group of $\ell$ inside
$G_\Q$. We will denote by $\FF$ a finite field of characteristic $p$ and by $W(\FF)$
its ring of Witt vectors.

By $\rho_n$ we will denote a continuous representation
\[
\rho_n:G_\Q \to \GL_2(W(\FF)/p^n).
\]
By $\trho$
we will always denote a continuous representation with coefficients in $W(\F)$ ramifying at 
finitely many primes and by $\rrho$ its reduction modulo $p$. If $\omega$ is a character from $G_\Q$ to
$\FF$, we denote by $\tilde{\omega}$ its Teichmuller lift.

We will denote by
$\chi$ the $p$-adic cyclotomic character. If $\det\rrho =
\omega\overline{\chi}^k$, with $\omega$ unramified at $p$, we will consider only
deformations with determinant $\tilde{\omega}\chi^k$. If $\rho$ is any continuous representation, we denote by $\Q(\rho)$ 
the field fixed by its kernel.

Given $\rrho$, after twisting it by a character of finite order we
may, and will, suppose that $\rrho$ and $\Ad$ ramify at the same set
of primes $S$. 

\medskip

\noindent {\bf Acknowledgments:} Special thanks go to Luis Dieulefait,
for proposing us the problem of Corollary \ref{coro:congruenceproblem}
(the starting point of the present article) as well as many
discussions and suggestions he made which improved the exposition, and
to Ravi Ramakrishna for many suggestions which not only improved the
exposition, but also allowed to remove some technical conditions in a
first version of the article. We also would like to thank Gabor Wiese
for many corrections and comments, and Panagiotis Tsaknias for
pointing out the application of Theorem A to
Corollary~\ref{coro:loweringthelevel}. Finally, we would like to thank
John Jones and Bill Allombert for helping us with the computational part of the example.

\section{Classification of residual representations and types of reduction}
Recall the classification of mod $p$ representations of $G_\ell$, when
$\ell \neq p$ (see  for example \cite{Diamond}, Section 2).
\begin{prop}
Let $\ell \neq 2$, be a prime number, with $\ell \neq p$. Then every
representation $\rho:G_\ell \rightarrow \GL_2(\F)$, up to twist by a
character of finite order, belongs to one of the following three
types:
\begin{itemize}
\item {\bf Principal Series:}
$\rho \simeq \left(\begin{smallmatrix}
\phi&0\\
0&1\\ 
\end{smallmatrix}\right) \hspace{0.25cm} \text{or} \hspace{0.25cm}
\rho \simeq \left(\begin{smallmatrix}
1&\psi\\
0&1\\
\end{smallmatrix}\right).$

\item {\bf Steinberg:} $\rho \simeq \left(\begin{smallmatrix}
\chi&\mu\\
0&1\\ 
\end{smallmatrix}\right),$ where $\mu \in \coho^1(G_\ell,\F(\chi))$ and 
$\mu|_{I_\ell} \neq 0$.
 
\item {\bf Induced:} $\rho \simeq \Ind_{G_M}^{G_\ell}(\xi)$, where
  $M/\Q_\ell$ is a quadratic extension and $\xi: G_M \rightarrow
  \F^\times$ is a character not equal to its conjugate under the
  action of $\Gal(M/\Q_\ell)$.
\end{itemize}
Here $\phi:G_\ell \to \F^\times$ is a multiplicative character and
$\psi:G_\ell \to \F$ is an unramified additive character.
\label{prop:Diamond}
\end{prop}
\begin{remark}
  Any unramified representation is Principal Series, and can be of the form
  $\rho \simeq \left(\begin{smallmatrix}
      \phi&0\\
      0&1\\
\end{smallmatrix}\right)$, with $\phi$ unramified or of the form $\rho \simeq \left(\begin{smallmatrix}
  1&\psi\\
  0&1\\
\end{smallmatrix}\right)$, with $\psi: G_\ell \to {\FF}$ an
additive unramified character.

\end{remark}
The same classification applies for representations $\tilde{\rho}:
G_\ell \rightarrow \GL_2(\overline{\Q_p})$, but since we need to study
reductions modulo powers of a prime, we need to look at
representations with integer coefficients modulo
$\GL_2(\overline{\Z_p})$ equivalence. Let $L$ be the coefficient field
of $\tilde{\rho}$, $\Om_L$ its ring of integers, and $\pi$ be a local
uniformizer. Also let $\mu \in \coho^1(G_\ell,\Z_p(\chi))$ be a
generator of such $\Z_p$-module.

\begin{prop}
  Let $\tilde{\rho}:G_\ell \rightarrow \GL_2(\overline{\Z_p})$ be a
  continuous representation. Then up to twist (by a finite order
  character times powers of the cyclotomic one) and
  $\GL_2(\overline{\Z_p})$ equivalence we have:
  \begin{itemize}
  \item {\bf Principal Series:} $\tilde{\rho} \simeq
    \left(\begin{smallmatrix}
        \phi & \pi^n(\phi-1)\\
        0 & 1\end{smallmatrix} \right)$, with $n \in \Z_{\leq 0}$ satisfying
    $\pi^{n}(\phi-1) \in \overline{\Z_p}$ or $\tilde{\rho} \simeq
    \left(\begin{smallmatrix}
        1&\psi\\
        0&1\\
\end{smallmatrix}\right)$.

\item {\bf Steinberg:} $\tilde{\rho} \simeq \left(\begin{smallmatrix}
      \chi & \pi^n\mu \\
      0 & 1\end{smallmatrix} \right)$, with $n \in \Z_{\geq 0}$.

\item {\bf Induced:} There exists a quadratic extension $M/\Q_\ell$
  and a character $\xi: G_M \rightarrow \overline{\Z_p}^\times$ not
  equal to its conjugate under the action of $\Gal(M/\Q_\ell)$ such
  that $\tilde{\rho} \simeq \<v_1,v_2>_{\O_L}$, where for $\sigma$ a
  generator of $\Gal(M/\Q_p)$ and $\tau \in G_M$, the action is given by
\[
  \tau(v_1)= \xi(\tau)v_1,\quad \tau(v_2)= \xi^\sigma(\tau)v_2, \quad  \sigma(v_1)= v_2 \quad \text{and} \quad
  \sigma(v_2)=\xi(\sigma^2)v_1,
\]
or
\[
  \tilde{\rho}(\tau)=\begin{pmatrix}
               \xi(\tau)&\frac{\xi(\tau)-\xi^\sigma(\tau)}{\pi^{n}}\\
               0&\xi^\sigma(\tau)
              \end{pmatrix}   \qquad \text{ and } \qquad \tilde{\rho}(\sigma)=\begin{pmatrix}
               -a&\frac{\xi(\sigma^2)-a^2}{\pi^n}\\
               \pi^n&a
              \end{pmatrix}  
\]

where $\xi^\sigma$ is the character of $G_M$ defined by $\xi^\sigma(g) = \xi(\sigma g\sigma^{-1})$ and $a \in \Om_L^\times$.

  \end{itemize}
\end{prop}

\begin{proof}

  We first consider the case where $\trho$ is irreducible over $\overline{\Q_p}$. 
  In this case the representation is induced, and in the
  coefficient field $L$, the canonical basis is $\{v_1,v_2\}$, where $v_2 =
  \sigma(v_1)$ for $\sigma$ a generator of $\Gal(M/\Q_\ell)$.  Let $T$
  be an invariant lattice for $\tilde{\rho}$. There exists a least $n
  \in \Z$ such that $w_1 = \pi^n v_1\in T$. Rescaling $T$ we can
  assume that $n=0$ (rescaling the lattice does not affect the
  representation). Since $\sigma(T) \subseteq T$, $v_2 = \sigma(v_1)
  \in T$. Since $\sigma(v_2) = \xi(\sigma^2)v_1$, with $\xi(\sigma^2)
  \in \O_L^\times$, $0$ is also the least integer such that $\pi^n v_2
  \in T$, and therefore $\langle v_1, v_2 \rangle_{\O_L} \subseteq
  T$. If this inclusion is an equality we are in the first case of our
  classification.
  
  Otherwise, we can extend $v_1$ to a basis of $T$ by adding a vector $w \in T$ such that 
  $w \notin \langle v_1, v_2 \rangle_{\O_L}$. We can write this element as $w = \alpha v_1 + \beta v_2$. Notice that 
  necessarily $v_\pi(\alpha) = v_\pi(\beta) < 0$. Changing $v_1$ and $v_2$ by a unit we can assume that $w = \pi^{-n}(-av_1+v_2)$, with $n<0$.
  Using $\sigma(v_1) = v_2$ and $\sigma(v_2) = \xi(\sigma^2)v_1$ we can compute the matrix of $\sigma$ in the basis $v_1,w$ and we get
  
  \[ \tilde{\rho}(\sigma) = \begin{pmatrix}
               -a&\pi^{-n}(\xi(\sigma^2)-a^2)\\
               \pi^n&a
              \end{pmatrix}.
  \]
  The action of inertia follows from a similar computation.


  On the other hand, ff $\tilde{\rho}$ is reducible over $\overline{\Q_p}$, 
  we can chose an eigenvector inside our lattice,
  and extend it to a basis so that our representation is of the form
  (up to twist)
\[
\tilde{\rho} \simeq \left( \begin{smallmatrix} \phi & * \\ 0 & 1\end{smallmatrix}\right).
\]
If $\phi$ is trivial, then $*$ is an additive character, and we are in
the first case. Otherwise, if $\tilde{\rho}$ is principal series, it
is equivalent (modulo $\GL_2(L)$) to $\left( \begin{smallmatrix}
    \phi & 0 \\ 0 & 1\end{smallmatrix} \right)$, hence is of the form
$\left( \begin{smallmatrix} \phi & u(\phi-1)\\ 0 & 1\end{smallmatrix}
\right)$. Since we want our representation to have integral
coefficients we get the stated result. Finally, in the Steinberg case,
our representation is $\GL_2(L)$-equivalent to
$\left(\begin{smallmatrix} \chi & \mu \\ 0 & 1\end{smallmatrix}
\right)$, but an easy computation shows that such a representation is
of the desired form as well.
\end{proof}

\begin{remark}
 In the Principal Series case, if we put $n = 0$ we get $\tilde{\rho} \simeq \left(\begin{smallmatrix}
                                                                                    \phi&\phi-1\\0&1
                                                                                   \end{smallmatrix}
\right)$, which is equivalent to $\left(\begin{smallmatrix}
                                                                                    \phi&0\\0&1
                                                                                   \end{smallmatrix}
\right)$, we will make repeated use of this last representative for this class.
\end{remark}

Now we want to study the possible reductions from types of $\GL_2(\overline{\Z_p})$-equivalent
representations to types of representations with coefficients in $\GL_2(\overline{\F_p})$. Although this is well known to
experts, and most of the claims are in \cite{carayol}, the change of
types are not explicitly described in that article, so we just give a
short self contained description.


Recall the condition for a character to lose ramification:

\begin{lemma}
  Let $\xi: G_\ell \rightarrow \overline{\Q_p}^\times$ a character and
  $\overline{\xi}$ its mod $p$ reduction. If $\Ker(\xi|_{I_\ell})
  \subsetneq \Ker(\overline{\xi}|_{I_\ell})$ then $\ell \equiv 1
  \hspace{0.1cm}(mod \hspace{0.1cm} p)$.
\label{lemma:1}
\end{lemma}







\begin{remark}
 Whenever an element $g\in I_\ell$ satisfies that $\xi(g) \neq 1$ and $\overline{\xi}(g) = 1$ we necessarily have $\xi(g)^{\ell-1} = 1$.
\end{remark}

\begin{prop}
Let $\trho$ be as above, then we have the following types of reduction:
\begin{itemize}
\item If $\trho$ is Principal Series, then $\rrho$ is Principal Series or Steinberg, and the latter occurs only when $\ell \equiv 1 \pmod p$.
\item If $\trho$ is Steinberg, then $\rrho$ is Steinberg or Principal Series, and the latter occurs only when $\rrho$ is unramified.
\item If $\trho$ is Induced, then $\rrho$ is Induced, Steinberg or an
  unramified Principal Series. For the last two cases we must have
  $\ell \equiv -1 \pmod p$.
\end{itemize}
\label{prop:localtypesreductions}
\end{prop}

\begin{proof}
  If $\trho$ is reducible, its reduction cannot be irreducible, which
  already excludes the case of a Principal Series or a Steinberg reducing to an Induced one. 
  Besides this trivial observation, we study each case in detail:

  \begin{itemize}
  \item \textsl{$\trho$ Principal Series:} in this case $\trho \simeq \left(\begin{smallmatrix}
\phi&\lambda(\phi -1)\\
0&1\\
\end{smallmatrix}\right)$ or $\left(\begin{smallmatrix}
1&\phi\\
0&1\\
\end{smallmatrix}\right)$.
If  $\trho \simeq \left(\begin{smallmatrix}
\phi&\lambda(\phi -1)\\
0&1\\
\end{smallmatrix}\right)$, the uniqueness of the semisimplification of the
reduction implies that $\rrho^{ss}
\simeq\left(\begin{smallmatrix}
    \overline{\phi}&0\\
    0&1
\end{smallmatrix}\right)$. If the reduction is of Steinberg type we need to have $\overline{\phi}
= \chi$, so a
character is losing ramification and this implies (by
Lemma~\ref{lemma:1}) that $\ell \equiv 1 \pmod{p}$.

If  $\trho \simeq \left(\begin{smallmatrix}
1&\phi\\
0&1\\
\end{smallmatrix}\right)$ then it is unramified and so is its reduction, implying that it can only be
Principal Series.

\item \textsl{$\trho$ Steinberg:} in this case $\trho
  \simeq \left(\begin{smallmatrix}
    \chi&\lambda u\\
    0&1\\
\end{smallmatrix}\right)$ where $u \in \coho^1(G_\ell, \Z_p(\chi))$ is the generator of the group.
Its semisimplification is
$\left(\begin{smallmatrix}
    \chi&0\\
    0&1
  \end{smallmatrix}\right)$, which implies that if $\rrho$ is Principal Series then it
is unramified.

\item \textsl{$\trho$ Induced:} in this case $\rho =
  \Ind_{G_M}^{G_{\Q_\ell}}(\xi)$, where $M/\Q_\ell$ is a quadratic
  extension and $\xi$ is a character of $G_M$ that does not descend to
  $G_{\Q_\ell}$. If the character $\overline{\xi}$ does not descend,
  then $\overline{\rho}$ is also irreducible hence Induced.

Now suppose that $\overline{\xi}$ does descend and, for a moment, that $\rrho$ ramifies (which implies, by assumption, that $\Ad$ ramifies). 
In this case the type of $\rho$ changes when reducing. The semisimplification of the reduction
we are considering is therefore $$\overline{\rho}^{ss}
\simeq \begin{pmatrix} \overline{\xi}\epsilon&0\\ 0&\overline{\xi}
\end{pmatrix} = \overline{\xi} \otimes \begin{pmatrix}
\epsilon&0\\
0&1
\end{pmatrix},$$
where $\epsilon$ is the quadratic character associated to $M/\Q_\ell$. 

Now, if $\overline{\rho}$ is Principal Series, then $\epsilon$ has to
be ramified, as we are assuming that $\Ad$ is ramified at $\ell$, so
$M/\Q_\ell$ is ramified. We claim (and will prove in the next
Lemma) that this case cannot happen, i.e. if $M/\Q_\ell$ is
ramified, any character $\xi:G_M \rightarrow \Z_p^\times$ that does
not extend to $G_\Q$ then its reduction does not extend to $G_\Q$ either.
Then the only possibility left to study is when $\overline{\rho}$ is
Steinberg. Observe that if this is the case, looking at the
semisimplifications we see that $\epsilon = \chi$, which only happens
when $M/\Q_\ell$ is unramified and $\ell = -1 \pmod p$. This finishes the case where $\rrho$ is ramified.

If $\rrho$ is unramified then $\epsilon$ has to be unramified as well, hence $M/\Q_\ell$ is an unramified extension. In this case, using the same
argument as in Lemma \ref{lemma:1}, we conclude that $\ell^2 \equiv 1 \pmod{p}$. It is easy to prove that if 
$\ell \equiv 1 \pmod{p}$ then the character $\xi$
extends to $G_\ell$, therefore we necessarily have $\ell \equiv -1 \pmod{p}$.
\end{itemize} 
\end{proof}

\begin{lemma}
  Let $M/\Q_\ell$ be a quadratic ramified extension and $\xi: G_M
  \rightarrow \overline{\Z_p}^\times$ a character and $\overline{\xi}$ its
  reduction. If $\overline{\xi}$ extends to $G_\ell$ then $\xi$ does as well.
\end{lemma}

\begin{proof}
  Let $L/\Q_p$ be a finite extension that contains the image of $\xi$,
  and $\pi$ an uniformizer of this extension. Let $\sigma \in G_\ell$
  be an element not in $G_M$ and define $\xi^\sigma(x) = \xi(\sigma
  x\sigma^{-1})$. We know that $\xi$ extends to $G_\ell$ if and only
  if $\xi = \xi^\sigma$.

  Via local class field theory, the character $\xi$ corresponds to a
  character $\psi$ defined over $M^\times$ and $\xi^\sigma$
  corresponds to $\psi^\sigma(x) = \psi(\sigma(x))$, so $\xi$ extends
  to $G_\ell$ if and only if $\psi$ factors through the norm map
  $N_{M/\Q_\ell}: M^\times \rightarrow \Q_\ell^\times$. Recall that by hypotheses 
  $\psi = \psi^\sigma \pmod \pi$ and
  we want to prove that $\psi = \psi^\sigma$. Let $\overline{\phi}$ be the
  factorization of $\overline{\psi}$ through the norm map.

If we restrict to the inertia subgroup we have the following picture:

$$
\xymatrix{
  \Ker \overline{\psi} \ar[d]_{N}\ar[rr]^{\psi|}\ar[dr]& &1+\pi\O_L \ar[d]\\
  \Ker\overline{\phi} \ar[dr]\ar@{-->}[urr]^{\phi|} &O_M^\times \ar[d]_{N}\ar[r]^{\psi}\ar[dr]_{\overline{\psi}}&\O_L^\times\ar[d]\\
  &\Z_\ell^\times \ar[r]_{\overline{\phi}}&\F_L^\times}$$
  
We are going to construct the dashed arrow $\phi|$ of the diagram above. Observe that $\psi|$ factors through $\Ker \overline{\psi}/(\Ker \overline{\psi} \cap (1+\ell\Z_\ell)) \subseteq \F_\ell^\times$ (since $1+\pi \O_L$ is a pro-p-group) so we have
\[
\xymatrix{ \Ker \overline{\psi} \ar[r]\ar[d]_{N} & \frac{\Ker \overline{\psi}}{\Ker \overline{\psi}\cap (1+\ell\Z_\ell)} \ar[r]^{\psi|} \ar[d]_{f} & 1+\pi O_L\\
           \Ker \overline{\phi} \ar[r] & \frac{\Ker \overline{\phi}}{\Ker \overline{\phi}\cap (1+\ell\Z_\ell)}\ar@{-->}[ur]^{\phi|}}
\]
where the down arrow $f$ is $f(x) = x^2$ (since $M/\Q_\ell$ is
ramified). So we can define the dashed arrow $\phi|$ as $\phi|(x) =
\sqrt{\psi|(x)}$ where $\sqrt{\;\; } : 1+\pi\O_L \rightarrow 1+\pi\O_L$
is the morphism that assigns to every $x \in 1 + \pi\O_L$ its square
root in $1 + \pi\O_L$ (which exists and is unique by Hensel's
Lemma). This makes the diagram commutative and proves that $\phi$ can
be extended in $\Ker \overline{\phi}$.

Now we want to prove that $\psi$ factors through the norm
map. Define $\tau(x) = \psi^\sigma\psi^{-1}$. We know that $\tau:
\O_M^\times \rightarrow 1+\pi\O_L$ and that $\tau(\Ker \overline{\xi})
= 1$. So it factors through $\overline{\tau}: \O_M^\times/\Ker
\overline{\psi} \rightarrow 1+\pi\O_L$, but $\O_M^\times / \Ker
\overline{\psi} \subseteq \F_L^\times$ and the only element of order
$p^n -1$ inside $1+\pi\O_L$ is $1$, so $\tau$ must be trivial and
therefore $\psi = \psi^\sigma$ when restricted to $\O_M^\times$. In
order to deduce $\psi = \psi^\sigma$ from this, we only need to check
it for the uniformizer, which is $\sqrt{\delta p}$ with $\delta = \pm
1$. We have:

$$\psi^\sigma(\sqrt{\delta p}) = \psi(\sigma(\sqrt{\delta p})) = \psi(-\sqrt{\delta p}) = \psi(-1)\psi(\sqrt{\delta p}) = \psi(\sqrt{\delta p}).$$

The last inequality follows from $\psi(-1) = \phi(N(-1)) = \phi(1) = 1$, because $-1 \in \O_M^\times$. We have proved that $\xi$ extends to $G_\ell$. \end{proof}




\begin{remark}
  Since we are only considering representations with unramified
  coefficient field, and $p\ge 5$, this rules out most change of type
  cases while reducing.
 \end{remark}

 \begin{prop}
   Let $\varrho: G_\ell \to \GL_2(W(\F))$ be a continuous representation.
   \begin{itemize}
   \item If $\varrho$ has type a ramified Principal Series
     then $\overline{\varrho}^{ss}$ is ramified.
   \item If $\varrho$ has type an Induced representation then
     $\overline{\varrho}^{ss}$ is ramified.
   \end{itemize}
\label{prop:integralreductiontypes}
 \end{prop}
\begin{proof}
  For the first case, assume that $\overline{\varrho}^{ss}$ is
  unramified. Then $\overline{\phi} = 1$ which by the remark following
  Lemma \ref{lemma:1} implies that $\ell \equiv 1 \pmod{p}$ and
  $\phi(\tau_\ell)$ has order a power of $p$. Therefore the
  eigenvalues of $\varrho(\tau_\ell)$ generate a totally ramified
  extension of $\Q_p$ of degree at least $p-1$, which is clearly
  impossible as they also have to satisfy a polynomial of degree $2$
  over some unramified extension of $\Q_p$ and $p > 3$.

  For the second one, assume that $\overline{\varrho}^{ss}$ is
  unramified. Then necessarily $\overline{\xi} =
  \overline{\xi^\sigma}$, implying that the character $\psi =
  \xi/\xi^\sigma$ loses all of its ramification when reduced. Again by
  the remark following Lemma \ref{lemma:1} this implies that
  $\psi(\tau_\ell)$ has order a power of $p$ implying that it
  generates a totally ramified extension of degree at least $p-1 > 2$. But
  $\psi(\tau_\ell)$ is the quotient between the eigenvalues of
  $\varrho(\tau_\ell)$, so it lies in an extension of degree $2$ of
  some unramified extension of $\Q_p$ which is absurd.
\end{proof}

\section{Local cohomological dimensions}

To apply Ramakrishna's method in our situation we need to compute
$d_i=\dim \coho^i(G_\ell, \Ad)$ for $i = 1,2$. The strategy in each case
is as follows: we first compute $d_0$ and $d_0^*$ (where $d_i^*=\dim
\coho^i(G_\ell, (\Ad)^*)$). By local Tate duality $d_2 = d_0^*$ and then
we can derive $d_1$ from the local Euler-Poincare characteristic
(which is zero). We do such computation in each case of the
classification of mod $p$ representations by choosing a good basis for each space.

\subsection*{Ramified Principal Series case:}
in this case we have $\overline{\rho} =\left( \begin{smallmatrix}
\phi&0\\
0&1\\
\end{smallmatrix}\right)$ with $\phi$ a ramified multiplicative
character. It easily follows that $\Ad \simeq \F(1) \oplus \F(\phi)
\oplus \F(\phi^{-1})$. As $\phi$ is ramified, $\F(\phi)$
(resp. $\F(\phi^{-1})$) is not isomorphic to $\F(1)$ nor $\F(\chi)$. So
we have two cases:

\begin{enumerate}
\item $\ell \equiv 1 \pmod p$ then  $d_0 = 1$, $d_2 = 1$ and therefore $d_1 = 2$.
\item $\ell \not \equiv 1 \pmod p$ then $d_0 = 1$, $d_2 = 0$ and therefore $d_1 = 1$.
\end{enumerate}
\subsection*{The Steinberg case:} in this case we need to do the
computations by hand. Considering the basis $\set{e_{01}, e_{10},
  e_{00} + e_{11}}$ of the space of matrices with trace zero and
explicitly computing the action of $\Ad$ on them, we derive the values
of the numbers $d_i$, which are:
\begin{enumerate}
\item If $\ell \equiv 1 \pmod p$ then $d_0 = 1$, $d_2 = 1$ and therefore $d_1 = 2$.
\item If $\ell \equiv -1 \pmod p$ then $d_0 = 0$, $d_2 = 1$ and therefore $d_1 = 1$.
\item If $\ell \not \equiv \pm 1 \pmod p$ then $d_0 = 0$, $d_2 = 0$ and therefore $d_1 = 0$.
\end{enumerate}
\subsection*{The Induced case:}

Recall the following Lemma (see \cite{ravi3}, Lemma 4)

\begin{lemma}
  Let $M/\Q_\ell$ a quadratic extension and $\overline{\rho}: G_\ell
  \longrightarrow \GL_2(\overline{\F_p})$ be twist-equivalent to
  $\Ind_{G_M}^{G_\ell} \xi$, with $\xi$ a character of $G_M$ which is
  not equal to its conjugate under the action of $\Gal(M/\Q_\ell)$.

Then $\Ad \simeq A_1 \oplus A_2$, with $A_i$ an absolutely irreducible $G_\ell$-module of dimension $i$ and $\coho^0(G_\ell, \Ad) = 0$. Moreover $\coho^2(G_\ell, \Ad) = 0$ unless $M/\Q_\ell$ is not ramified and $\ell \equiv -1 \hspace{0.1cm} (mod \hspace{0.1cm} p)$ in which case it is one dimensional.
\end{lemma}

So for the Induced case we have two possibilities:

\begin{enumerate}
\item If $\ell \equiv -1 \pmod p$ and $M/\Q_\ell$ is unramified then $d_0 =
  0$, $d_2 = 1$ and therefore $d_1 = 1$.
\item If $\ell \not \equiv -1 \pmod p$ or $M/\Q_\ell$ is ramified then $d_0 = 0$, $d_2 = 0$ and therefore $d_1 = 0$.
\end{enumerate}

\subsection*{Unramified case:} if $\rrho$ is unramified, we
consider the following three cases according to the image of
Frobenius:

\begin{enumerate}
 \item $\rrho (\Frob_p)= \left(\begin{smallmatrix}
                              1&0\\
                              0&1
                             \end{smallmatrix}\right)$.
 \smallskip

\noindent In this case $\Ad \simeq \F^3$ thence we have two possibilities:

\begin{itemize}
\item $\ell \equiv 1 \pmod p$ then  $d_0 = 3$, $d_2 = 3$ and therefore $d_1 = 6$.
\item $\ell \not \equiv 1 \pmod p$ then $d_0 = 3$, $d_2 = 0$ and therefore $d_1 = 3$.
\end{itemize}\medskip

 \item $\rrho(\Frob_p)= \left(\begin{smallmatrix}
                              \alpha&0\\
                              0&1
                             \end{smallmatrix}\right)$ with $\alpha \not \equiv 1 \pmod p$.
\smallskip

\noindent We have that $\Ad \simeq \F\oplus \F(\phi) \oplus \F(\phi^{-1})$, with $\phi \neq 1$ and $\phi = \chi$ only if $\alpha \equiv \ell \pmod p$.
Again, we need to distinguish between cases:

\begin{itemize}
 \item $\ell \equiv -1 \pmod p$ and $\ell \equiv \alpha, \alpha^{-1} \pmod p$ then $d_0 = 1$, $d_2 = 2$ and therefore $d_1 = 3$.
 \item $\ell \equiv -1 \pmod p$ and $\ell \not \equiv \alpha, \alpha^{-1} \pmod p$ then $d_0 = 1$, $d_2 =0 $ and therefore $d_1=1$.
 \item $\ell \not \equiv -1 \pmod p$ and $\ell \equiv \alpha, \alpha^{-1}$ or $1 \pmod p$ then $d_0 = 1$, $d_2 =1$ and therefore $d_1=2$.
 \item $\ell \not \equiv -1 \pmod p$ and $\ell \not \equiv \alpha, \alpha^{-1}$ or $1 \pmod p$ then $d_0 = 1$, $d_2 =0$ and therefore $d_1=1$.

\end{itemize}

\medskip

 \item[(c)] $\rrho(\Frob_p)= \left(\begin{smallmatrix}
                              1&1\\
                              0&1
                             \end{smallmatrix}\right)$.
\smallskip

\noindent Here we do the computations by hand and establish that:

\begin{itemize}
 \item If $\ell \equiv 1 \pmod p$ then $d_0 = d_2 = 1$ and therefore $d_1 = 2$.
 \item If $\ell \not \equiv 1 \pmod p$ then $d_0 = 1$, $d_2 = 0$ and therefore $d_1 = 1$.
\end{itemize}

\end{enumerate}

\section{The sets $C_\ell$}

In order to apply Ramakrishna's method we need to define for each
prime $\ell \in P$ a set $C_\ell$ of deformations of $\overline{\rho}$ 
(containing $\rho_\ell$) and a subspace $N_\ell \subseteq \coho^1(G_\ell, \Ad)$ of dimension $d_1 - d_2$
such that $\overline{\rho}$ can be successively deformed to an
element of $C_\ell$ by deforming from $W(\F)/p^s$ to $W(\F)/p^{s+1}$
with adjustments at each step made only by a multiple of an element $h \notin N_\ell$.
In order to get the full statement of our theorem, we have to take the 
extra care of picking the set $C_\ell$ such that all its elements agree up to
isomorphism in the inertia group with $\rho_\ell$.

Notice that it is enough to do this for one representative of each of the possible 
types of $\GL_2(\overline{\Z_p})$-equivalence for $\rho_\ell$, as we can always pick 
a basis for $\rho_n$ for which it is the reduction of one of those representatives.
The only extra care we need to take is making sure that whenever we pick a set $\C_\ell$, 
the deformations that belong to it have all coefficients in $W(\F))$ and not in a bigger
extension of $\Q_p$. The potential issue that this may bring is that sometimes we cannot
use the representatives of $\GL_2(\overline{\Z_p})$-equivalence classes we defined above 
and need to translate our calculations to $W(\F)$.


We classify the selection of the sets $C_\ell$ according to the type of $\rrho$, considering
for each one, all the possible types for $\rho_\ell$.

%
%
%
%
%
%
%

\medskip

\noindent \textbf{Case 1: $\overline{\rho}$ is ramified Principal Series.} When $\overline{\rho}$
is ramified Principal Series, we have seen that $\rho_\ell$ can only be Principal
Series. Nevertheless, the cohomology groups are different depending on
whether $\ell \equiv 1 \pmod p$ or not.  Recall that the representatives for the equivalence classes were
 $\rho_\ell \simeq \left(\begin{smallmatrix}
  \phi& \pi^n(\phi -1)\\
  0&1
\end{smallmatrix}\right)$ with $n\leq 0$ such that $\pi^n(\phi -1)$ lies in $\overline{\Z_p}$. Observe that 
if $n \neq 0$, then $\pi \mid (\phi-1)$ and therefore its reduction is
not ramified Principal Series (the residual case $\left(\begin{smallmatrix}
  1& *\\
  0&1
\end{smallmatrix}\right)$ is unramified or Steinberg according to our classification).
Then $\rho_\ell \simeq \begin{pmatrix}
  \phi&0\\
  0&1
\end{pmatrix}$ over $\GL_2(\overline{\Z_p})$ and we have the following cases:

\begin{enumerate}
\item If $\ell \not \equiv 1 \mod p$, $d_0 = d_1 = 1$
and $d_2 = 0$ so we must take $N_\ell = \coho^1(G_\ell, \Ad)$ the full
cohomology group so there is no possible choice at each step and
$C_\ell$ must be the full set of deformations to characteristic
zero. Notice that this is the only possible choice whenever $d_2 = 0$ and $\ell \neq p$
and in this case we have to check that any lift of $\overline{\rho}$ to
$W(\F)/p^s$ is the reduction of a characteristic zero one, but this is
automatic as $d_2 = 0$ so the problem is unobstructed. 

In order to check that all the elements of $C_\ell$ agree up to
isomorphism when restricted to $I_\ell$, we need to describe the set $C_\ell$.
If we define a morphism $n : G_\ell \rightarrow
G_\ell/I_\ell \simeq \hat{\Z} \rightarrow \Z/p\Z$, then the element
\[
h(g) = \begin{pmatrix}
n(g)&0\\
0&-n(g)
\end{pmatrix}
\] 
generates $\coho^1(G_\ell, \Ad)$ and this implies that every lift is
Principal Series, as the set $\lambda h \cdot \psi_s$, where $\psi$ is
the Teichmuller lift of $\overline{\rho}$ and $\lambda$ is a scalar,
exhausts all the possible reductions. In particular, the restriction
to inertia is the same for all of them.

\item If $\ell \equiv 1 \pmod p$ the picture is slightly different
  since $d_0=1$, $d_1=2$ and $d_2=1$, so we need to choose a one
  dimensional subspace $N_\ell$ and a set of deformations $C_\ell$ to
  $W(\F)$. Observe that the isomorphism between $\rho_\ell$ and the
  representative of its $\GL_2(\overline{\Z_p})$-equivalence class may
  not realize over $W(\F)$.

  If the image of $\psi_1$ lies in $W(\F)$, then the isomorphism
  does realize over $W(\F)$.  In that case, observe that the element $h$
  defined above lies inside $\coho^1(G_\ell, \Ad)$. Let $N_\ell =
  \langle h \rangle$, and $C_\ell = \left\lbrace \begin{pmatrix}
      \psi_1\gamma&0\\
      0&\psi_2\gamma^{-1}
\end{pmatrix}: \gamma \hspace{0.1cm} \text{unramified character} \right\rbrace$. 

We claim that this choice verifies the hypotheses. Clearly $\rho_\ell
\in C_\ell$, and given any $h' \notin N_\ell$, the full
$\coho^1(G_\ell, \Ad)$ is generated by $h$ and $h'$. Then for any mod
$p^s$ deformation $\tilde{\rho}$ of $\rrho$ there is an element
$\lambda_1h + \lambda_2h' \in \coho^1(G_\ell, \Ad)$ such that
$(\lambda_1h + \lambda_2h')\tilde{\rho}$ lies in $C_\ell$. But the
action of any multiple of $h$ preserves the elements of $C_\ell$, so
$\lambda_2h'\tilde{\rho}$ already lies in $C_\ell$. Note that as in
the previous case, all the elements in $C_\ell$ have the same
restriction to inertia.\\

If the image of $\psi_1$ does not lie in $W(\F)$ then $\rho_\ell$ is not
isomorphic to $\left(\begin{smallmatrix}
    \psi_1&0\\
    0&\psi_2
  \end{smallmatrix} \right)$ over $W(\F)$ and we cannot use the
previous choice. Instead, we need to use a canonical form for $\rho_\ell$
over $W(\F)$. Assume that $\psi_1(\sigma_\ell) = \alpha$
and $\psi_2(\sigma_\ell) = \beta$, then the matrix $C =
\left(\begin{smallmatrix}
    -\beta& -\alpha\\
    1&1
\end{smallmatrix} \right)$ 
conjugates $\left(\begin{smallmatrix}
    \psi_1(\sigma_\ell)&0\\
    0&\psi_2(\sigma_\ell)
\end{smallmatrix} \right)$ into 
$\left(\begin{smallmatrix}
    0&-\alpha \beta\\
    1& \alpha + \beta
\end{smallmatrix} \right) \in \GL_2(W(\F))$. Therefore we can assume (applying a change of basis)
that $\rho_\ell(\sigma_\ell) = \left(\begin{smallmatrix}
    0&-\alpha \beta\\
    1& \alpha + \beta
  \end{smallmatrix} \right) $.  Then we can essentially use the same
sets and subspaces as in the previous case but conjugated by
$C$.

Let $N_\ell = \langle (\alpha - \beta)C h C^{-1}\rangle$, where $h$ is
the element defined before, and $C_\ell$ the set of deformations to
$W(\F)$ of the form $C\left(\begin{smallmatrix}
    \psi_1\gamma&0\\
    0&\psi_2\gamma^{-1}
\end{smallmatrix}\right)C^{-1}$
with $\gamma:G_\ell \to \overline{\Z_p}$ an unramified character. The
factor $\alpha - \beta$ forces the element generating $N_\ell$ to have
coefficients in $W(\F)$.

It can be easily checked that whenever $\trho$ is the reduction of
some element in $C_\ell$ and $u \in N_\ell$ then $(1+p^nu)\trho$ is
again the reduction of an element of $C_\ell$. Therefore the same
reasoning as before shows that $N_\ell$ and $C_\ell$ satisfy our
hypotheses.


\end{enumerate}

\begin{remark}
  Whenever we construct a set $C_\ell$ and subspace $N_\ell$ such that
  $N_\ell$ preserves the reductions of $C_\ell$ (i.e. whenever $\trho$
  is the reduction of some element of $C_\ell$ and $u \in N_\ell$,
  $u\cdot \trho$ is reduction of some element of $C_\ell$ as well) the
  proof is exactly the same. In the next cases the same phenomena will occur.
  
\end{remark}

\medskip

\noindent \textbf{Case 2: $\overline{\rho}$ is Steinberg.}  If
$\overline{\rho}$ is of Steinberg type then Proposition
\ref{prop:localtypesreductions} and Proposition~\ref{prop:integralreductiontypes} imply that $\rho_\ell$ can only be Steinberg.


\begin{enumerate}
\item If $\ell \not \equiv \pm 1 \pmod p$, by the previous section results, $d_0 = d_1 =
d_2 = 0$, implying there is only one deformation at each $p^n$. We take $C_\ell = \set{\rho_\ell}$.

\item If $\ell \equiv -1 \pmod p$, by the previous section results, $d_1 =
d_2 = 1$ and $d_0 = 0$, so $N_\ell = \set{0}$ and we have the full
$\coho^1(G_\ell, \Ad)$ available to adjust at every step. Then we take
$C_\ell = \set{\rho_\ell}$.

\item If $\ell \equiv 1 \pmod p$, we take the element $j \in \coho^1(G_\ell, \Ad)$ given by
$0$ at the wild inertia subgroup and by
\[
j(\sigma_\ell) = \begin{pmatrix}
0&1\\
0&0
\end{pmatrix}\hspace{0.2cm} \text{,} \hspace {0.2cm}j(\tau_\ell) = \begin{pmatrix}
0&0\\
0&0
\end{pmatrix},\] where $\sigma_\ell$ is a Frobenius element and
$\tau_\ell$ a tame inertia generator (recall these two generate
$G_\ell/W_\ell$, where $W_\ell$ is the wild inertia, subject to the
relationship $\sigma\tau\sigma^{-1} = \tau^\ell$).  Let $N_\ell =
\langle j \rangle$ and $C_\ell$ the set of lifts $\rho$ satisfying
\[
\rho(\sigma_\ell) = \begin{pmatrix}
\ell&*\\
0&1
\end{pmatrix} \hspace{0.2cm} \text{and} \hspace{0.2cm} \rho(\tau_\ell) = \begin{pmatrix}
1&*\\
0&1
\end{pmatrix}.
\]
This set contains all the extensions $\tilde{\rho_\ell}$ of
$\rho_\ell$ to the decomposition group, and $N_\ell$ preserves its
reductions.

\end{enumerate}

\medskip

\noindent \textbf{Case 3: $\overline{\rho}$ is Induced.}  If $\overline{\rho}$ is
Induced then the only possibility for $\rho_\ell$ is also being of
Induced type.
\begin{enumerate}
\item If $\ell \equiv 1 \pmod p$ and $M/\Q_\ell$ is unramified,
  $d_0=0$, $d_1=d_2=1$ so $N_\ell$ is of codimension $1$ inside a
  space of dimension $1$, hence $N_\ell = \set{0}$. We take 
  $C_\ell = \set{\rho_\ell}$. Since we can adjust at every step by a
  multiple of a given element $h \notin \set{0}$, and $d_1 = 1$, we
  can adjust at each step by any element of $\coho^1(G_\ell, \Ad)$ to modify
  $\rho_n$ as we want.
\item If $\ell \not \equiv 1 \pmod p$ or $M/\Q_\ell$ is ramified, $d_0
  = d_1 = d_2 = 0$, so there is only one lift at every step. This lift
  must be the reduction of $\rho_\ell$, so there is nothing to adjust.
\end{enumerate}

\medskip

\noindent \textbf{Case 4: $\overline{\rho}$ is unramified.}  We need
to define the sets $C_\ell$ for the primes at which $\rho_n$ ramifies
and $\rrho$ does not.  By
Proposition~\ref{prop:integralreductiontypes} this can only happen
when $\rho_\ell$ is Steinberg.

\medskip

We have that $\rho_\ell =\left( \begin{smallmatrix}
\chi&*\\
0&1\\
\end{smallmatrix}\right)$, 
with $*|_{I_\ell} \neq 0 \pmod{p^n}$. The sets $C_\ell$ we will pick depend on the image of $\sigma_\ell$. Recall
that the eigenvalues of $\rrho(\sigma_\ell)$ are $1$ and $\ell$.  
\begin{enumerate}
\item If  $\rrho(\sigma_\ell) =\left( \begin{smallmatrix}
1&0\\
0&1\\
\end{smallmatrix}\right),$
necessarily $\ell \equiv 1 \pmod p$ implying $d_1 = 6$ and $d_2 = 3$
and therefore we need a subspace of dimension $3$, preserving a family
of deformations $C_\ell$. In the previous cases, we have built sets
$C_\ell$ of deformations of $\rho_n$ that depend on $d_2 - d_1$
parameteres, which in this case does not seem to be possible. However,
as pointed to us by Ravi Ramakrishna, one can construct elements which
are not cohomological trivial for the residual representation, but
give isomorphic lifts modulo big powers of $p$, as in Section $4$ of
\cite{ravihamblen}. Let $C_\ell$ be the set of deformations of
$\rho_n$ satisfying:

\[
\rho(\sigma_\ell) = \begin{pmatrix}
\ell&*\\
0&1
\end{pmatrix} \hspace{0.2cm} \text{and} \hspace{0.2cm} \rho(\tau_\ell) = \begin{pmatrix}
1&*\\
0&1
\end{pmatrix}.
\]

Observe that this family depends on two parameters and is clearly preserved by the
elements $u_1, u_2 \in \coho^1(G_\ell, \Ad)$ given by

\[
u_1(\sigma_\ell) = \begin{pmatrix}
0&1\\
0&0
\end{pmatrix}\hspace{0.2cm} \text{,} \hspace {0.2cm}u_1(\tau_\ell) = \begin{pmatrix}
0&0\\
0&0
\end{pmatrix},\]

and

\[
u_2(\sigma_\ell) = \begin{pmatrix}
0&0\\
0&0
\end{pmatrix}\hspace{0.2cm} \text{,} \hspace {0.2cm}u_2(\tau_\ell) = \begin{pmatrix}
0&1\\
0&0
\end{pmatrix}.\]

We still need one more element of $\coho^1(G_\ell, \Ad)$ to preserve $C_\ell$. Recall that $\rho_n$ 
satisfies

\[
\rho_n(\sigma_\ell) = \begin{pmatrix}
\ell&x\\
0&1
\end{pmatrix} \hspace{0.2cm} \text{and} \hspace{0.2cm} \rho_n(\tau_\ell) = \begin{pmatrix}
1&y\\
0&1
\end{pmatrix},
\]
with $y \neq 0$. There exists an element $v \in \coho^1(G_\ell, \Ad)$  that satisfies
   that whenever $\rho_m$ is the reduction modulo $p^m$ of some element in $C_\ell$ then $(1+p^{m-1}v)\rho_m$ 
is the same deformation as $\rho_m$. The element $v$ will depend on the valuations of $x$, $y$ and
$\ell-1$. As we mentioned in the introduction of this Section, we only need to do this for $m \geq n+1$.

\begin{lemma}
\label{lemma:v}
 There exists an element $v \in \coho^1(G_\ell, \Ad)$ such that whenever $\rho_m$ is the reduction modulo 
 $p^m$ of some element in $C_\ell$, with $m \geq n+1$, then $(1+p^{m-1}v)\rho_m$ 
is the same deformation as $\rho_m$.
\end{lemma}

\begin{proof}
 The proof is divided into several cases, we first define $g_1, g_2, g_3 \in \coho^1(G_\ell, \Ad)$ as
 
 \[
g_1(\sigma_\ell) = \begin{pmatrix}
0&0\\
1&0
\end{pmatrix}\hspace{0.2cm} \text{,} \hspace {0.2cm}g_1(\tau_\ell) = \begin{pmatrix}
0&0\\
0&0
\end{pmatrix},\]
\vspace{0,5cm}
 \[
g_2(\sigma_\ell) = \begin{pmatrix}
1&0\\
0&-1
\end{pmatrix}\hspace{0.2cm} \text{,} \hspace {0.2cm}g_2(\tau_\ell) = \begin{pmatrix}
0&0\\
0&0
\end{pmatrix},\]

and

\[
g_3(\sigma_\ell) = \begin{pmatrix}
0&0\\
0&0
\end{pmatrix}\hspace{0.2cm} \text{,} \hspace {0.2cm}g_3(\tau_\ell) = \begin{pmatrix}
1&0\\
0&-1
\end{pmatrix}.\]

We now enumerate a list of cases (depending on the valuations of $x,
y$ and $\ell-1$) and for each of them specify an element $v$ and a
matrix $C$ congruent to the identity modulo $p$ such that
$C^{-1}\rho_m C = (1+p^{m-1})\rho_m$. Write $C = \left(\begin{smallmatrix}
  1+p\alpha&p\beta\\
  p\gamma&1+p\delta \end{smallmatrix}\right)$. In each case we will give the values of $\alpha,
\beta, \gamma$ and $\delta$ and left to the reader to check that
$C^{-1}\rho_m C = (1+p^{m-1})\rho_m$ in each of them.
                                                    
\begin{itemize}
\item If $v_p(y) < v_p(x)$ and $v_p(y) < v_p(\ell-1)$: take $v = g_3$
  and $C$ satisfying $\alpha = \delta$, $\beta = 0$, $\gamma y =
  p^{m-2} \pmod{p^{m-1}}$ and $\gamma x = \gamma(\ell-1) = 0
  \pmod{p^{m-1}}$.  
\item If $v_p(x) < v_p(y)$ and $v_p(x) < v_p(\ell-1)$: take $v = g_2$
  and $C$ satisfying $\alpha = \delta$, $\beta = 0$, $\gamma x =
  p^{m-2} \pmod{p^{m-1}}$ and $\gamma y = \gamma(\ell-1) = 0
  \pmod{p^{m-1}}$.  
\item If $v_p(\ell-1) < v_p(x)$ and $v_p(\ell-1) < v_p(x)$: take $v =
  g_1$ and $C$ satisfying $\alpha = \delta$, $\beta = 0$, $\gamma
  (\ell -1) = -p^{m-2} \pmod{p^{m-1}}$ and $\gamma x = \gamma y = 0
  \pmod{p^{m-1}}$.
 
 \item If $v_p(y) = v_p(\ell-1)$ and $v_p(y) < v_p(x)$:
then $y = \lambda(\ell-1)$. Take $v = g_1 - \lambda g_3$ and $C$ satisfying $\alpha = \delta$, $\beta = 0$, 
 $\gamma (\ell -1) = -p^{m-1} \pmod{p^{m-1}}$ and $\gamma x= 0 \pmod{p^{m-1}}$.
 
\item If $v_p(y) = v_p(x)$ and $v_p(y) < v_p(\ell-1)$: then $y =
  \lambda x$. Take $v = g_2 + \lambda g_3$ and $C$ satisfying $\alpha
  = \delta$, $\beta = 0$, $\gamma x = p^{m-2} \pmod{p^{m-1}}$ and
  $\gamma(\ell -1)= 0 \pmod{p^{m-1}}$.
 
 \item If $v_p(x) = v_p(\ell-1)$ and $v_p(x) < v_p(y)$: then $x = \lambda(\ell-1)$. Take $v = g_1 - \lambda g_2$ and $C$ satisfying $\alpha = \delta$, $\beta = 0$, 
 $\gamma (\ell -1) = -p^{m-2} \pmod{p^{m-1}}$ and $\gamma y= 0 \pmod{p^{m-1}}$.
 
\item If $v_p(x) = v_p(\ell-1) = v_p(y)$: then $x = \lambda_1
  (\ell-1)$ and $y = \lambda_2 (\ell-1)$. Take $v = g_1 - \lambda_1g_2
  - \lambda_2 g_3$ and $C$ satisfying $\alpha = \delta$, $\beta = 0$,
  $\gamma (\ell -1) = -p^{m-2} \pmod{p^{m-1}}$.
\end{itemize}
\end{proof}
We end this case by taking $C_\ell$ as above and $N_\ell = \langle
u_1, u_2, v \rangle$, for the element $v$ of Lemma \ref{lemma:v}.

\item If $\rrho(\sigma_\ell) =\left( \begin{smallmatrix}
\alpha&0\\
0&1\\
\end{smallmatrix}\right)$, with $\alpha\neq 1$,
necessarily $\ell \equiv \alpha \pmod p$ so $d_1 = 3$ and $d_2 = 2$ if $\ell \equiv -1 \pmod{p}$ and $d_1 = 2$ and $d_2 =
1$ otherwise. In both cases, let $u \in \coho^1(G_\ell, \Ad)$ defined by $u(\sigma_{\ell}) = \left(\begin{smallmatrix}
    0&0\\
    0&0 \end{smallmatrix}\right)$
and $u(\tau_\ell) = \left(\begin{smallmatrix}
    0&1\\
    0&0 \end{smallmatrix}\right)$, and take $N_\ell = \<u>$. Define the
set $C_\ell$ of deformations $\rho$ that satisfy
\[
\rho(\sigma_\ell) = \rho_\ell(\sigma_\ell) \hspace{0.5cm} \text{and} \hspace{0.5cm} \rho(\tau_\ell) = \begin{pmatrix}
  1&*\\
  0&1
\end{pmatrix}.
\]
Clearly $N_\ell$ preserves $C_\ell$.
\item If  $\rrho(\sigma_\ell) =\left( \begin{smallmatrix}
1&1\\
0&1\\
\end{smallmatrix}\right),$
necessarily $\ell \equiv 1 \pmod p$, so $d_1 = 2$ and $d_2 = 1$. Let
$u \in \coho^1(G_\ell, \Ad)$ by $u(\sigma_\ell) = 0$ and $u(\tau_\ell)
= \left(\begin{smallmatrix}
    0&1\\
    0&0 \end{smallmatrix}\right)$ and take $N_\ell = \langle u
\rangle$. This subspace preserves the set $C_\ell$ of
deformations $\rho$ satisfying

\[
\rho(\sigma_\ell) = \rho_\ell(\sigma_\ell) \hspace{0.5cm} \text{and} \hspace{0.5cm} \rho(\tau_\ell) = \begin{pmatrix}
 1&*\\
 0&1
\end{pmatrix}.
\]
\end{enumerate}

\begin{remark}
  If we allow ramification in the coefficient field then the cases
  ruled out by Proposition~\ref{prop:integralreductiontypes} may
  happen.  Most of them correspond to cases like the first unramified
  case, where a trick like in \cite{ravihamblen} need to be used. It
  is worth pointing out that in such cases we can construct the
  corresponding sets $C_\ell$ and subspaces $N_\ell$ but the global
  arguments below do not adapt well to that situation. See the remark
  after Lemma~\ref{Lemma:intersec}.
\end{remark}

%

\subsection{The case $\ell = p$}

In this case we will pick $C_p$ exactly as in \cite{ravi3}
(\emph{local at $p$ considerations}), with the observation that in the
supersingular case, it follows from the work done in \cite{ravi1} that
the lifts picked have the same Hodge-Tate weights than $\rho_p$
(which lie in the interval $[0,p-1]$) and are crystalline.
Note that in each case considered by Ramakrishna, $\rho_p$
is always trivially contained in $C_p$.

\section{Auxiliary primes}

For constructing the sets $Q_1$ and $Q_2$ mentioned in the introduction we
will work with primes $q \not\equiv \pm 1 \pmod p$ such that $\rrho$
is not ramified at $q$ and $\rrho(q)$ has different eigenvalues of
ratio $q$, i.e. $\bar{\rho}(\sigma_q) = \left(\begin{smallmatrix} qx &
    0\\ 0 & x\end{smallmatrix}\right)$ and $\bar{\rho}(\tau_q) =
\left(\begin{smallmatrix} 1 & 0\\ 0 & 1\end{smallmatrix}\right)$.
 For these primes the cohomological dimensions are $\dim
\coho^0(G_q,\Ad) = 1$, $\dim \coho^1(G_q,\Ad) = 2$ and $\dim
\coho^2(G_q,\Ad) = 1$.

In this case, the set $C_q$ is
formed by the deformations $\omega$ such that
\begin{equation} \label{aux} \omega(\tau_q) = \begin{pmatrix}
  1&px\\
  0&1\\
  \end{pmatrix} \hspace{0.25cm} \text{and} \hspace{0.25cm}
\omega(\sigma_q) = \begin{pmatrix}
q&py\\
0&1\\
\end{pmatrix}.
\end{equation}
 
These two conditions define a tamely ramified deformation of
$\rrho$. The set $C_q$ is preserved by a subspace $N_q \subseteq
\coho^1(G_q,\Ad)$ of codimension $1$ given by
$j(\sigma_q)=\left(\begin{smallmatrix} 0&0\\0&0\end{smallmatrix}
\right)$ and $j(\tau_q)=\left(\begin{smallmatrix}0 & 1\\ 0 &
    0\end{smallmatrix}\right)$.

There are two main goals we want to achieve in this section. Firstly,
we would like to prove that auxiliary primes do exist for
representations $\rho$ with coefficients in $W(\F)/p^n$. Observe that,
the inductive step depends only on the reduction modulo $p$ of $\rho$,
so we only need to check that once we set the deformation set $C_q$,
whenever we add an auxiliary prime $q$ together with its subspace
$N_q$, the representation $\rho_n|_{G_q}$ is the reduction of some
element in $C_q$, i.e. we want to prove that there are primes $q$ such
that $\rho_n|_{G_q}$ sends a Frobenius and a generator of the tame
inertia to the matrices defined in (\ref{aux}) modulo $p^n$.

Secondly, we need to reprove the properties of the auxiliary primes
we are going to use in our context, although they look similar to the
arguments in \cite{ravi3}.

\subsection{Working modulo $p^n$}

We need to prove that there exist infinitely many auxiliary primes,
that is primes $q$ such that $q \not\equiv \pm1 \pmod p$, $\rho_n$
is unramified at $q$ and $\rho_n(\Frob_q)$ has different eigenvalues of
ratio $q$. 

Following \cite{ravi2} and \cite{ravi3}, let $\mu_p$ be a primitive
$p$-th root of unity, $D = \Q(\Ad)\cap\Q(\mu_p)$, $K =
\Q(\Ad)\Q(\mu_p)$, $D' = \Q(\Add_n)\cap\Q(\mu_p)$ and $K' =
\Q(\Add_n)\Q(\mu_p)$, which fit in the following diagram:

\[
\xymatrix@=1em{
 &K'\ar@{-}[dl]\ar@{-}[dr]&  & & &K\ar@{-}[dl]\ar@{-}[dr]&  \\
\Q(\Add_n)\ar@{-}[dr]&  &\Q(\mu_p)\ar@{-}[dl] & & \Q(\Ad)\ar@{-}[dr]&  &\Q(\mu_p)\ar@{-}[dl]\\
  &D'\ar@{-}[d]  &  & &  &D\ar@{-}[d]  &  \\		
  &\Q  &  & &   &\Q    &                  \\}
\]
Observe that we can translate the conditions on $q$ into the following:
\begin{itemize}
\item the condition $q \not \equiv \pm1 \pmod p$ is equivalent to
  $\Frob_q$ not being the identity nor conjugation in $\Gal(\Q(\mu_p)/\Q)$.
\item $q$ being an auxiliary prime is equivalent to being unramified
  in $\Q(\Add_n)$, $\Frob_q \not \equiv \pm 1 \pmod p$ and $\Frob_q$
  lies in the conjugacy class of an element $\overline{M} \in
  \Img(\Add)_n$, where $M$ is a diagonal matrix with eigenvalues of
  ratio $q$.
\end{itemize}
Therefore, if we prove that there is an element $\sigma \in
\Gal(K'/\Q)$ such that $\sigma|_{\Gal(\Q(\mu_p)/\Q)} = t \neq \pm 1$
and $\sigma|_{\Gal(\Q(\Add_n)/\Q)} = \overline{M}$ where $M$ is
diagonal with eigenvalues of ratio $t$, then we are done
using Chebotarev's Theorem.

\begin{prop}
  There is an element $c=a\times b \in
  \Gal(\Q(\Add_n)/D')\times\Gal(\Q(\mu_p)/D') \simeq \Gal(K'/D')$ such
  that $a$ comes from an element $M \in \Img(\rho_n) \simeq
  \Gal(\Q(\rho_n)/\Q)$ which has different eigenvalues with ratio $b \in
  \F_p^\times \simeq \Gal(\Q(\mu_p)/\Q)$, $b \neq \pm 1$.
\label{prop:elementc}
\end{prop}
The proof is in the spirit of the arguments given in \cite{ravi2} for finding
such elements. Recall the following lemma (Lemma 3, IV-23 in \cite{serreladic} \footnote{Actually, Lemma 3 is stated and proved in \cite{serreladic} for $\F = \F_p$ but the same proof holds for an arbitrary finite field of characteristic $p$.})

\begin{lemma}
 Let $p\geq 5$ and $\F$ a finite field of characteristic $p$. Let $H\subseteq \GL_2(W(\F))$ a closed subgroup and $\overline{H}$ its projection to $\GL_2(\F)$. If
$\SL_2(\F) \subseteq \overline{H}$ then $\SL_2(W(\F)) \subseteq H$.
\end{lemma}

This has the following easy consequences:

\begin{coro}
 If $\SL_2(\F) \subseteq \Img(\rrho)$ then $\SL_2(W(\F)/p^n) \subseteq \Img(\rho_n)$.
\label{coro:bigimage}
\end{coro}
\begin{proof}
  Denote by $\pi: W(\F) \rightarrow W(\F)/p^n$ the projection, then
  this follows applying the above lemma with $H =
  \pi^{-1}(\Img(\rho_n)) \subseteq W(\F)$ which is closed as $G_\Q$ is
  compact .
\end{proof}


The following lemma gives the existence of the element $c$.

\begin{lemma}
For $D'$ the field defined above, $[D':\Q] \leq 2$. 
\end{lemma}

\begin{proof}
\label{Dprima}

Observe that $[\Q(\Add_n):\Q(\Ad)] = p^*$ which is coprime with $[\Q(\mu_p):\Q]$. 
This implies that $D' = \Q(\Add_n) \cap \Q(\mu_p) = \Q(\Ad) \cap \Q(\mu_p)$ and 
$[\Q(\Ad) \cap \Q(\mu_p):\Q] = 1$ or $2$ by Lemma $18$ of \cite{ravi2}

%
\end{proof}
\begin{proof}[Proof of Proposition~\ref{prop:elementc}:]
  Let $b \in \F_p^\times \subseteq \F^\times$ be any element such that
  $b^2 \neq \pm 1$. Let $\tilde{b} \in \set{1,\cdots,p-1} \subseteq
  W(\F)/p^n$ be congruent to $b$ modulo $p$ and $M
  = \begin{pmatrix}
    \tilde{x}&0\\
    0&\tilde{x}^{-1}\\
  \end{pmatrix} \in \SL_2(W(\F)/p^n) \subseteq \Img(\rho_n)$. Then $c
  = (\overline{M},b^2) \in \Gal(\Q(\Add)/D')\times\Gal(\Q(\mu_p)/D')$
  is such an element.
\end{proof}

\begin{remark}
  The element $c$ constructed in Proposition~\ref{prop:elementc} is
  not the same as the one in \cite{ravi2}. In fact they live in
  different Galois groups, the first one lying in $\Gal(K'/\Q)$ and
  the second one in $\Gal(K/\Q)$.  However, it is true that the
  projection of the element constructed in this work through the map $\Gal(K'/\Q) 
  \to \Gal(K/\Q)$ is an element like the one defined by Ramakrishna. 
  In particular, both elements act in the same way on $\Ad$ (as the
  action of our $c$ is through this projection). To avoid confusion
  we denote the projection by $\tilde{c}$ .
\end{remark}

Any prime $q$ not ramified in $K'$ such that $\Frob_q$ lies in the
conjugacy class of $c$ can be taken as an auxiliary prime. In the
next subsection we are going to impose extra conditions at the
auxiliary primes regarding their interaction with elements of
$\coho^1(G_\Q, \Ad)$ and $\coho^2(G_\Q,\Ad)$.

\subsection{Properties of auxiliary primes}

We need to impose conditions to the auxiliary primes similar to the ones in Fact
16 and Lemma 14 of \cite{ravi3}. Concretely, for non-zero elements $f
\in \coho^1(G_P,\Ad)$ and $g \in \coho^1(G_P,(\Ad)^*)$, the auxiliary prime
$q$ should satisfy $f|_{G_q} = 0$ or $f|_{G_q} \notin N_q$ and
$g|_{G_q}\neq0$. We need to impose these conditions for many
elements at the same time.

If $f \in \coho^1(G_P,\Ad)$, then $f|_{\Gal(\overline{\Q}/\Q(\Ad))}$ is a
morphism, so we can associate an extension $\widetilde{L_f}/\Q(\Ad)$
fixed by its kernel. Also let $L_f = \widetilde{L_f}K =
\widetilde{L_f}(\mu_p)$.  Analogously, for $g \in \coho^1(G_P,(\Ad)^*)$ we
define $M_g/\Q((\Ad)^*)$ as the fixed field by the kernel of
$g|_{\Gal(\overline{\Q}/\Q((\Ad)^*))}$.  Notice that we can obtain
information about $f|_{G_q}$ or $g|_{G_q}$ by looking at the conjugacy
class of $\Frob_q$ in $\Gal(L_f/\Q)$ or $\Gal(M_g/\Q)$ (as these are
almost the extensions associated to the adjoint representation of
$\rrho(Id + \epsilon f)$).

Let $f_1, \ldots, f_{r_1}$ and $g_1, \ldots, g_{r_2}$ basis for
$\coho^1(G_P, \Ad)$ and $\coho^1(G_P, (\Ad)^*)$ respectively. Define $L$ to be
the composition of the fields $L_{f_i}$, $M$ the composition of the
$M_{g_j}$, and $F = LM$. The following lemma is a summary of results
about these extensions from \cite{ravi2}.

\begin{lemma} Let $f_i$ and $g_j$ as above.

\begin{enumerate}
\item For every $f_i$, $\Gal(L_{f_i}/K) \simeq \Ad$ as $G_\Q$-modules,
  and for every $g_j$, $\Gal(M_{g_j}/K) \simeq (\Ad)^*$.
\item $\Gal(L/K) \simeq \prod \Gal(L_{f_i}/K) \simeq (\Ad)^{r_1}$ and
  $\Gal(M/K) \simeq \prod \Gal(M_{g_j}/K) \simeq
  ((\Ad)^*)^{r_2}$. Also $M \cap L = K$ so $\Gal(F/K) \simeq
  \Gal(L/K)\times\Gal(M/K)$.
\item The exact sequences
\[
1\longrightarrow \Gal(L/K) \longrightarrow \Gal(L/\Q) \longrightarrow \Gal(K/\Q) \longrightarrow 1,
\]
and
\[
1\longrightarrow \Gal(M/K) \longrightarrow \Gal(M/\Q) \longrightarrow \Gal(K/\Q) \longrightarrow 1,
\]
both split, hence $\Gal(F/\Q) \simeq \Gal(F/K)\rtimes \Gal(K/\Q)$.
\end{enumerate}
\end{lemma}

\begin{proof}
  The first claim is Lemma $9$, the second is Lemma $11$ and the last
  one is Lemma $13$ of \cite{ravi2} with two remarks:
  \begin{itemize}
  \item In \cite{ravi2} these results are proved for the
    representation $\Adr$, which is the descent of $\Ad$ to its
    minimal field of definition. As we are assuming that $\SL_2(\F)
    \subseteq \Img(\rrho)$, we have that $\Ad$ is already defined in
    its minimal field of definition, because of Lemma $17$ of \cite{ravi2}.
  \item In \cite{ravi2} these lemmas are proved for $P = S$ the set of
    ramification of $\Ad$, but the same proofs work for any $P
    \supseteq S$.
  \end{itemize}
\end{proof}

Finally, we can read properties of $f|_{G_q} \in \coho^1(G_q, \Ad)$ from
the class of $\Frob_q$ in $\Gal(L_f/\Q) \simeq \Gal(L_f/K) \rtimes
\Gal(K/\Q)$. Observe that the element $c \in \Gal(K'/\Q)$ constructed
in the previous section acts on $\Ad$ through the projection to
$\Gal(\Q(\Ad)/\Q)$.

\begin{prop} Let $q \in \Q$ be a prime, $f\in \coho^1(G_P, \Ad)$ and $g
  \in \coho^1(G_P, (\Ad)^*)$.
 \begin{enumerate}
 \item If $\Frob_q$ lies in the conjugacy class of $1 \rtimes
   \tilde{c} \in \Gal(L_f/\Q)$
   then $f|_{G_q} = 0$. The same holds for $g$ and $\Gal(M_g/\Q)$.
 
 \item There are nontrivial elements $\alpha \in \Ad$ on which $c$
   acts trivially and if $\Frob_q$ lies in the conjugacy class of
   $\alpha \rtimes \tilde{c} \in \Gal(L_f/\Q)$ then $f|_{G_q} \notin
   N_q$.
 
 \item There are nontrivial elements $\beta \in (\Ad)^*$ on which $c$
   acts trivially and if $\Frob_q$ lies in the conjugacy class of
   $\beta \rtimes \tilde{c} \in \Gal(M_g/\Q)$ then $g|_{G_q} \neq 0$.
 \end{enumerate}
\label{prop:auxiliaryprimes}
\end{prop}

\begin{proof}
  See Lemmas 14, 15 and 16, and Corollaries 1 and 2 of \cite{ravi2},
  noting that in our setting $\Ad = \Adr$, so the proof of the
  existence of $\alpha$ and $\beta$ is almost trivial.
\end{proof}

\begin{coro}
  There exists primes $q$ such that $\rrho(\Frob_q)$ has different
  eigenvalues of ratio $q$ and such that for the basis elements any of
  the following conditions can be achieved: $f_i|_{G_q} = 0$ or
  $f_i|_{G_q} \notin N_q$ and $g_j|_{G_q} = 0$ or $g_j|_{G_q} \neq 0$.
\label{coro:setQ}
\end{coro}
\begin{proof}
Pick an element
\[
\Omega = \omega \rtimes \tilde{c} \in \Gal(F/\Q) \simeq \left(\prod_{i=1}^{r_1} \Gal(L_{f_i}/\Q) \times \prod_{j=1}^{r_2} \Gal(M_{g_j}/\Q)\right) \rtimes \Gal(K/\Q),
\]
where $\omega$ has coordinates $0$ or $\alpha$ whether we want $f_i|_{G_q}$ to be $0$ or not in $N_q$ in the first product and $0$ or $\beta$ 
whether we want $g_j|_{G_q}$ to be $0$ or not $0$ in the second one. Then any $q$ such that $\Frob_q$ lies in the conjugacy class of $\Omega$ works.
\end{proof}

We want the same to hold for $\rho_n$, i.e. to find primes $q$
satisfying the same conditions plus $\rho_n(\Frob_q)$ to have
different eigenvalues of ratio $q$. As we mentioned before, any $q$
such that $\Frob_q \in \Gal(K'/\Q)$ lies in the conjugacy class of $c$
satisfies this extra condition. Therefore, we only need to check that
there is an element $\theta$ in $\Gal(K'F/\Q)$ such that $\theta|_{K'}
= c$ and $\theta|_{F} = \Omega$.

Observe that $\Omega|_K = \tilde{c} = c|_K$, a necessary condition. It
is enough to prove that $K' \cap F = K$, as any pair of
elements in $\Gal(K'/\Q)$ and $\Gal(F/\Q)$ that are equal when
restricted to $K'\cap F$ define an element in $\Gal(K'F/\Q)$.
In order to prove this, we need the following lemma.

\begin{lemma}
\label{Lemma:intersec}
$K'\cap F = K$.
\end{lemma}

\begin{proof}

  Let $\HH  = \Gal(K'/K) \subseteq \PGL_2(W(\F)/p^n)$ and $\pi_1:
  \PGL_2(W(\F)/p^n) \to \PGL_2(\F)$.  Observe that $H$ consists on the
  classes of matrices in $\Img(\rho_n)$ which are trivial in
  $\PGL_2(\F)$, i.e. $\HH = \Img(\Add_n) \cap \Ker(\pi_1)$.  Recall that
  our hypotheses imply $\PSL_2(W(\F)/p^n) \subseteq \Img(\Add_n)
  \subseteq \PGL_2(W(\F)/p^n)$, and therefore $\PSL_2(W(\F)/p^n)\cap
  \Ker(\pi_1) \subseteq \HH \subseteq \Ker(\pi_1)$. As
  $[\PSL_2(W(\F)/p^n):\PGL_2(W(\F)/p^n)] = 2$ and $\Ker(\pi_1)$ is a
  $p$ group we have that $\HH = \Ker(\pi_1)$.

  Recall that $\Gal(F/K) \simeq (\Ad)^r \times (\Ad^*)^s$ as
  $\Z[G_\Q]$-module and by Lemma $7$ of \cite{ravi2}, this is
  its decomposition as $\Z[G_\Q]$ simple modules. This implies that
  if $K' \cap F \neq K$ then $\Ad$ or $(\Ad)^*$ appear as a quotient
  of $\Gal(K'/K)$.

  Assume that $K' \cap F \neq K$ and that there is a surjective
  morphism $\varpi: \HH \to \Ad$. Let $\pi_2: \PGL_2(W(\F)/p^n) \to
  \PGL_2(W(\F)/p^2)$ and let $\LL = \ker(\pi_2) \subset \HH$. We claim that
  $\varpi(\LL) = 0$.  For this, observe that any matrix $\text{Id} + p^2M
  \in \GL_2(W(\F)/p^n)$ is the $p$-th power of some matrix $\text{Id} +
  pN \in \GL_2(W(\F)/p^n)$. 
Therefore, if $\text{Id} + p^2M \in \LL$ we have that 
\[
\varpi(Id + p^2M) = \varpi((Id + pN)^p) = p\varpi(Id + pN) = 0.
\]
This implies that $\varpi$ factors through $\Gal(\Q(\Add_2)/K)$, where $\Add_2$ is the reduction mod $p^2$ of 
$\Add_n$.  Since $\#\Gal(\Q(\Add_2)/K) = \#(\Img(\Add_2)\cap \Ker(\pi_1)) \leq (\#\F)^3$ and
$\#\Ad = (\#\F)^3$ we necessarily have
$\Gal(\Q(\Add_2)/\Q) = \Gal(L_f/\Q)$ for some $f \in \coho^1(G_\Q, \Ad)$. But this 
cannot happen since it would imply that the image of $\Add_2$ splits, which is impossible as
it contains $\PSL_2(W(\F)/p^2)$ when $p \geq 7$ or $\PGL_2(W(\F)/p^2)$
when $p=5$.

The case where there is a surjection $\pi: \HH \to (\Ad)^*$ works the same.
\end{proof}

\begin{remark}
  As we mentioned before, this global argument does not adapt to the
  cases when the coefficient field is ramified. Specifically, Lemma
  \ref{Lemma:intersec} above in no longer true if we allow the
  coefficients to ramify, as the extension corresponding to $\Add_2$
  corresponds to an element of $\coho^1(G_\Q,\Ad)$. Then we cannot
  apply Chebotarev's Theorem to find auxiliary primes which are
  nontrivial in the element of the cohomology corresponding to
  $\Add_2$, so we do not get an isomorphism between local and global
  deformations. 
\end{remark}

\begin{prop}
\label{surj1}
 For any $\tau \in \Gal(L/K)$ as above we have that
 \[
\coho^1(G_{P \cup T_\tau}, \Ad) \longrightarrow \bigoplus_{\ell \in P} \coho^1(G_\ell, \Ad)
\]
is a surjection.
 
\end{prop}

\begin{proof}
  This is essentially Proposition $10$ of \cite{ravi3}, up to the fact
  that we ask a condition on $\Gal(K'/\Q)$ rather than
  $\Gal(K/\Q)$. Nevertheless, the same proof applies as the main
  argument is that for any $g \in \coho^1(G_{P \cup T_\tau},(\Ad)^*)$
  there are primes $q \in T_\tau$ such that $g|_{G_q} \neq 0$ and this
  is Proposition~\ref{prop:auxiliaryprimes}.
 \end{proof}

\section{Proof of main theorems}
\begin{thmA}
  Let $\F$ be a finite field of characteristic $p>5$.  Consider
  $\rho_n: G_\Q \to \GL_2(W(\F)/p^n)$ a continuous representation
  ramified at a finite set of primes $S$ satisfying the following properties:
  \begin{itemize}
  \item The image is big, i.e. $\SL_2(\F) \subseteq \Img(\rrhon)$.
  \item $\rho_n$ is odd.
  \item The restriction $\overline{\rho_n}|_{G_p}$ is not twist equivalent
    to the trivial representation nor the indecomposable unramified
    representation given by $\left(\begin{smallmatrix} 1 &*\\ 0 &
        1\end{smallmatrix}\right)$.


  \end{itemize}
Let $P$ be a finite set of primes containing $S$, and for every
  $\ell\in P$, $\ell \neq p$, fix a deformation 
$\rho_\ell:G_\ell \to W(\F)$ of $\rho_n|_{G_\ell}$.
At the prime $p$, let $\rho_p$ be
a deformation of $\rho_n|_{G_p}$ which is ordinary or crystalline with
Hodge-Tate weights $\{0,k\}$, with $2 \le k \le p-1$.

Then there is a finite set $Q$ of auxiliary primes $q \nequiv \pm1 \pmod{p}$ and a modular representation
\[
\rho: G_{P\cup Q} \longrightarrow \GL_2(W(\F)),
\]
such that:
\begin{itemize}
\item the reduction modulo $p^n$ of $\rho$ is $\rho_n$,
\item $\rho|_{I_\ell} \simeq \rho_\ell|_{I_\ell}$ for every $\ell \in P$,
\item $\rho|_{G_q}$ is a ramified representation of Steinberg type for every $q \in Q$.
\end{itemize}
 \end{thmA}

\begin{proof}
  Once we have all the ingredients, the proof mimics that of Theorem 1
  of \cite{ravi3}. Let $r=
  \dim_\FF\sha_P^2(\Ad)=\dim_\FF\sha^1_P((\Ad)^*)$, and let
  $\{g_1,\ldots,g_r\}$ be a basis of $\sha^1_P((\Ad)^*)$. Let
  $\{f_1,\ldots,f_r\}$ be a linearly independent set in
  $\coho^1(G_P,\Ad)$. For each $i=1,\dots,r$ let $q_i$ be such that:
\[
f_i|_{G_{q_i}}\notin N_{q_i}, \qquad g_i|_{G_{q_i}}\neq 0 \qquad f_j|_{G_{q_i}}=g_j|_{G_{q_i}}=0 \text{ for }j\neq i.
\]
Such primes exists in virtue of Corollary~\ref{coro:setQ} and
Lemma~\ref{Lemma:intersec}.  Let $Q_1=\{q_1,\ldots,q_r\}$ so that
$\sha_{P\cup Q_1}^2(\Ad)=0=\sha_{P\cup Q_1}^1((\Ad)^*)$. With this choice,
the inflation map $\coho^1(G_P,\Ad)\to \coho^1(G_{P\cup Q_1},\Ad)$ is an
isomorphism by the same dimension counting as in the proof of Fact 16
(\cite{ravi3}). As mentioned in the introduction, we need to pick the
set of primes $Q_2$ such that the map
\[
\coho^1(G_{S\cup Q_1 \cup Q_2},\Ad) \to \bigoplus_{\ell \in S \cup Q_1 \cup Q_2} \coho^1(G_\ell,\Ad)/N_{\ell},
\]
is an isomorphism. Recall that once we achieved $\sha^2_{P\cup Q_1} = 0$, no set of extra primes we consider adds new global obstructions.

The way to construct such set is as follows: take a basis
$\{f_1,\ldots,f_d\}$ of the preimage under the restriction map
$\coho^1(G_P,\Ad) \to \oplus_{l \in P}\coho^1(G_\ell,\Ad)$ of the set
$\oplus_{\ell \in P}N_\ell$. By Lemma 12 (\cite{ravi3}), $r \ge
d$. For $r+1\le i \le d$, let $\alpha_i$ be an element of $\Gal(L/K)$
all whose entries are $0$ except the i-th which is a nonzero element
in which $\tilde{c}$ acts trivially. By Proposition~\ref{surj1}, the
map $\coho^1(G_{S\cup Q_1\cup T_i},\Ad) \to \oplus_{\ell \in
  P}\coho^1(G_\ell,\Ad)$ is surjective. By Lemma 14 (\cite{ravi3}), we
can pick a prime ideal $\id{p}_i \in T_i$ such that if
$T=\{\id{p}_{r+1},\ldots,\id{p}_d\}$, then the map
\[
\coho^1(G_{P\cup Q_1 \cup Q_2},\Ad)\to \oplus_{\ell \in P} \coho^1(G_\ell,\Ad)/N_\ell,
\]
is surjective. The same proofs of Lemma 15 and 16 (\cite{ravi3}) show
that this set $Q_2$ satisfies the required properties.
This proves the existence of the lift, the condition on the
restriction to inertia is automatic by the choice of the sets
$C_\ell$.

To prove that $\rho$ is modular, we know it has big residual image
hence it is residually modular (by Serre's conjectures). The modularity
is covered by the following two modularity lifting theorems: for the
ordinary case modularity follows as a consequence of Theorem $5.4.2$
of \cite{gera} (precisely we are in a situation covered by the
consequence stated as theorem in the introduction); for the
supersingular case we apply Theorem 3.6 of \cite{diamond3}. Observe
that $\rho$ is crystalline by definition and meets the shortness
condition because it preserves the Hodge-Tate weights of $\rho_{f,p}$,
which satisfy $2\leq k \leq p-1$. The irreducibility condition holds
because of the big image hypothesis.

\end{proof}

Let us recall the hypothesis of our second result: let $f\in
S_k(\Gamma_0(N),\epsilon)$ be a newform, with coefficient field
$K$. Let $\id{p}$ a prime ideal in $\Om_K$ dividing $p$ with
ramification index $1$ and let 
\[
\rho_n:\Gal(\overline{\Q}/\Q) \to \GL_2(\Om_{\id{p}}/{\id{p}^n}),
\] 
the reduction modulo $\id{p}^n$ of its $p$-adic Galois representation.
 
\begin{thmB}
  In the above hypothesis, let $n$ be a positive integer and $p>k$ be
  a prime such that:
  \begin{itemize}
  \item $p\nmid N$ or $f$ is ordinary at $p$,
  \item $\SL_2(\Om_{\id{p}}/\id{p}) \subseteq \Img(\overline{\rho_{f,p}})$,
  \end{itemize}
  Let $N'$ denote the conductor of $\rho_n$. If $N''=\prod_{p \mid
    N'}p^{v_p(N)}$, i.e. $N''$ is $N$ divided by its prime-to-$N'$
  part, then there exist an integer $r$, a set $\{q_1,\ldots,q_r\}$ of
  auxiliary primes prime to $N$ satisfying $q_i \not \equiv 1 \pmod p$
  and a newform $g$, different from $f$, of weight $k$ and level
  $N''q_1\ldots q_r$ such that $f$ and $g$ are congruent modulo
  $p^n$. Furthermore, the form $g$ can be chosen with the same
  restriction to inertia as that of $f$ at the primes dividing $N'$.
\end{thmB}

\begin{proof}
  We want to apply Theorem A to the representation $\rho_n$, with the
  local deformation $\rho_{f,p}|_{I_\ell}$ at the primes dividing
  $N'$. Note that $f$ being a modular form implies that the
  representation is odd, and the hypothesis $p>k$ implies that
  $\rho_{f,p}|_{I_p}$ satisfies the third hypothesis of such
  theorem. Finally, the condition $p\nmid N$ or $f$ being ordinary at
  $p$ implies that $\rho_{f,p}|_{I_p}$ can be taken as a deformation
  at $p$.

  Theorem A then gives a modular representation $\rho$ which is
  congruent to $\rho_{f,p}$ modulo $p^n$, and of conductor dividing
  $N'q_1\dots q_r$. By the choice of the inertia action, the conductor
  of $\rho$ has the same valuation as the $\rho_n$ one at the primes
  dividing $N'$, so we only need to show that all the primes $q_i$ are
  ramified ones. But if this is not the case, by the choice of the
  sets $C_{q_i}$, and looking at the action of Frobenius, it would
  contradict Weil's Conjectures, since the roots of the Frobenius'
  characteristic polynomial would be $1$ and $q$, which do not have
  the same absolute value.
  
  Note that when $\rho_f$ does not lose ramification when reduced 
  modulo $p^n$ and $r = 0$, the newform $g$ that Theorem A produces could
  be equal to $f$. If this is the case, we apply Theorem A with $P = S \cup \{q\}$, $q$
  being in the hypotheses of auxiliary primes and $$\rho_q = \begin{pmatrix}
                                                             \chi&*\\
                                                             0&1
                                                            \end{pmatrix}$$
  with $*$ ramified (up to twist).

\end{proof}

\section{Example}

We want to apply the main result to some particular example. More
concretely, we want to add some Steinberg primes to a modular form,
modulo powers of a prime. For that purpose we pick the smallest prime
in the hypothesis, $p=5$, and start with a representation coming from
an elliptic curve $E$ of prime level $\id{q}$ (in order to deal with
small cohomological dimensions) with full image modulo $5$,
i.e. $\Gal(\Q(E[5])/\Q) \simeq \GL_2(\F_5)$. Its adjoint
representation is then isomorphic to $\PGL_2(\F_5)$ which is
isomorphic to $S_5$, the symmetric group in $5$ elements. For
$S=\{5,\id{q}\}$, we need to compute $\coho^1(G_{S},\Ad)$ and
$\coho^2(G_S,\Ad)$. Recall the following dimension computations:

\begin{itemize}
\item If $\ell \not\equiv \pm 1 \pmod p$ then $\coho^2(G_\ell,\Ad)=0$ (see
  Section 3, or \cite{ravi2} Proposition 2).
\item Suppose that at $p$ inertia acts via fundamental characters of level two. Then $\coho^2(G_p,\Ad)=0$ (\cite{ravi2} Lemma 5).
\item Suppose that $\bar{\rho}$ is flat, and $\bar{\rho}|_{G_p}$ is indecomposable. Then $\coho^2(G_p,\Ad)=0$ (\cite{ravi1}).
\end{itemize}

Also, if we denote by $r = \dim\sha_S^1((\Ad)^*)$, and $s$ the number
of primes with $H^2(G_\ell,\Ad)\neq 0$, then (see \cite{ravi3} Lemma, page 139):
\begin{itemize}
\item $\dim \coho^1(G_S,\Ad)=r+s+2$.
\item $\dim \coho^2(G_S,\Ad)=r+s$.
\end{itemize}

\subsection{Some group theory} 
Recall from Lemma 9 (of \cite{ravi2}) that the elements in
$\coho^1(G_s,\Ad)$ (resp. in $\coho^1(G_s,\Ad^*)$) correspond to
extensions $M$ of $\Q(\Ad)$ (resp. $\Q(\Ad^*)$) whose Galois group is
isomorphic to $\PGL_2(\F_5) \ltimes M_2^0(\F_5)$ (the $2\times 2$
matrices with zero trace), so we need to compute all such
extensions. The problem is that $\PGL_2(\F_5)$ has order $120$, and we
cannot do Class Field Theory in such a huge extension, so we will
reduce the problem to compute some abelian extension over a small
degree extension of $\Q$ where we actually can compute Class Field Theory.

The action of $S_5$ in $M_2^0(\F_5)$ is faithful, so we need to
restrict the action to smaller subgroups to find the desired extension.

\begin{lemma}
  Let $H$ be a subgroup of $S_5$, and suppose that the restriction of
  the action of $S_5$ in $M_2^0(\F_5)$ to $H$ decomposes as the direct
  sum of two subspaces $V_1 \oplus V_2$. Then $H \ltimes V_i$ is a
  subgroup of $S_5 \ltimes M_2^0(\F_5)$. Furthermore, if $V_1$ is one
  dimensional, then $H \ltimes V_2$ is a normal subgroup of $H \ltimes
  M_2^0(\F_5)$ if and only if $V_1$ is the trivial representation.
\end{lemma}

\begin{proof}
  The first claim is clear from the definition of a semi-direct
  product. For the second claim, let $\{v_1,v_2,v_3\}$ be a basis of
  $M_2^0(\F_5)$ such that $V_1 = \<v_1>$ and $V_2=\<v_2,v_3>$. Then it
  is clear that $H \ltimes V_2$ is invariant under elements of the
  form $(h,v_i)$ with $i=2,3$ (since it is a subgroup), and since it
  is enough to check invariance on generators it is enough to check
  invariance under elements of the form $(h,v_1)$. But a direct
  computation shows that
\[
(h,v_1)(g,w)(h,v_1)^{-1}=(hgh^{-1},v_1+h\cdot w-(hgh^{-1})\cdot v_1),
\]
which lies in $V_2$ if and only if $g\cdot v_1 = v_1$ for all $g \in H$.
\end{proof}

Then we need a subgroup of $S_5$ whose order is prime to $5$ (for the
representation to be semisimple), whose restriction contains the
trivial representation and such that the intersection of its
conjugates is trivial (for the Galois closure of the fixed field to be
the whole extension). The subgroups of $S_5$ of order prime to $5$
are: $\{1\}$, $C_2$, $C_2 \times C_2$, $C_4$, $D_8$, $C_3$, $C_6$,
$S_3$, $S_3\times C_2$, $A_4$, $S_4$ (where $C_n$ means a cyclic group
of order $n$, and $D_n$ the dihedral group with $n$ elements). The
largest one (in terms of cardinality) for which the actions splits is
$S_3 \times C_2$, for which $M_2^0(\F_5)$ splits as a direct sum 
\[
<(3,1,0),(3,0,1)> \oplus \<(4,1,1)>,
\]
via the identification $S_3\times C_2 =
\<\left(\begin{smallmatrix}1&2\\2&0\end{smallmatrix}\right),
\left(\begin{smallmatrix}4&2\\1&1\end{smallmatrix}\right)> \times
  \<\left(\begin{smallmatrix}3&2\\2&2\end{smallmatrix}\right)>$ in
  $\PGL_2(\F_5)$ and in the basis of $M_2^0(\F_5)$
  $\left\{\left(\begin{smallmatrix}1 &0 \\0
        &4 \end{smallmatrix}\right),\left(\begin{smallmatrix}0 &1 \\0
        &0 \end{smallmatrix}\right),\left(\begin{smallmatrix}0 &0 \\1
        &0 \end{smallmatrix}\right)\right\}$ .  The action in the
  $1$-dimensional subspace is non-trivial, nevertheless the
  restriction to its cyclic subgroup of order $6$ is trivial as can be
  seen via a direct computation (and actually such group is the stabilizer of the matrix $\left(\begin{smallmatrix}4 & 1\\1& 1\end{smallmatrix}\right)$). It is clear that
  the intersection of its conjugates is trivial (since $A_5$ is the
  only normal subgroup of $S_5$ and the action of $S_5$ in
  $M_2^0(\F_5)$ is irreducible).

\begin{lemma}
  $(C_3 \times C_2) \ltimes V_2 \lhd (S_3 \times C_2) \ltimes M_2^0(\F_5).$
\label{lemma:abelian}
\end{lemma}

\begin{proof}
  The previous Lemma implies that $(C_3 \times C_2) \ltimes V_2 \lhd
  (C_3 \times C_2) \ltimes M_2^0(\F_5)$ but since $C_3 \lhd S_3$, the
  same proof gives the statement.
\end{proof}

Then we first search for the $S_5$ extension corresponding to the
adjoint representation (which might be given as the Galois closure of
a degree $5$ extension) and then we search for the fixed field of
$(C_3\times C_2) \ltimes M_2^0(\F_5)$, which is a degree $20$
extension of $\Q$. By Lemma~\ref{lemma:abelian}, the field fixed of
$(C_3\times C_2) \ltimes V_2$ is a degree $5$ abelian extension $L_2$
of it, so we can compute it using class field theory. Note that since
$(S_3 \times C_2)\ltimes V_2$ is a subgroup, the degree five extension
we are looking for actually is a non-Galois degree $5$ extension $L_1$ of the
degree $10$ extension over $\Q$ fixed by $(S_3 \times C_2)\ltimes
M_2^0(\F_5)$. We illustrate this phenomena in the following diagram:
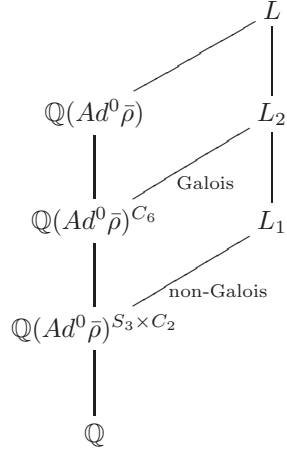
\[
\xymatrix{
& L \ar@{-}[dl]\ar@{-}[d]\\
\Q(\Ad) \ar@{-}[d] & L_2\ar@{-}[d]\ar@{-}[dl]^{\text{Galois}}\\
\Q(\Ad)^{C_6}\ar@{-}[d] & L_1\ar@{-}[dl]^{\text{non-Galois}}\\
\Q(\Ad)^{S_3\times C_2}\ar@{-}[d]\\
\Q }
\]

To compute with the adjoint representation, we must add the 
$5$-th roots of unity. The Hasse diagram is the following
\[
\xymatrix@=1em{
& \Q(\Ad,\xi_5) \ar@{-}[dl] \ar@{-}[dr]\\
\Q(\Ad)\ar@{-}[dr] & & \Q(\xi_5)\ar@{-}[dl]\\
& \Q(\sqrt{5})\ar@{-}[d]\\
& \Q
}
\]

The Galois group $\Gal(\Q(\Ad^*)/\Q) \simeq \Gal(\Q(\Ad,\xi_5)/\Q)
\simeq C_4 \ltimes A_5$, where the action is through the projection
$C_4 \to C_2$, and the latter action is the classical isomorphism $S_5
\simeq C_2 \ltimes A_5$. This Galois group also acts on $M_2^0(\F_5)$,
where the $C_4$ part acts as $\F_5^\times$ (which corresponds to the
mod $5$-cyclotomic character action), and $A_5$ as before. To compute
the Shafarevich group $\sha^1(G_S,\Ad^*)$, we do a similar trick as
before, we consider the subgroup $C_4 \ltimes C_3$ (which also
satisfies that the intersection of its conjugates is trivial), which
is an extension of the previous cyclic group of order $6$, and get
exactly the same degree $20$ extension.

\subsection{Particular example}
Consider the elliptic curve
\[
E_{17a1}: y^2+xy+y=x^3-x^2-x-14.
\]

The representation obtained by looking at the $5$-torsion points has
full image (using \cite{sage}), so is isomorphic to
$\GL_2(\F_5)$. Then its adjoint representation corresponds to a Galois
extension of $\Q$ with Galois group isomorphic to $\PGL_2(\FF_5)
\simeq S_5$ and only ramified at $5$ and $17$. We can search for such
extensions (they are the Galois closure of a degree $5$ extension)
in Roberts-Jones tables (see \cite{jones}), and get $12$ such
extensions, given by the polynomials:
\begin{eqnarray*}
x^5 + x^3 - 2x^2 - 2x - 3, & x^5 - 5x^2 + 5, & x^5 - 85x - 153, \\
 x^5 - 15x^3 - 75x^2 - 110x - 89, & x^5 - 50x^2 + 100x - 65, & x^5 - 10x^3 - 20x^2 - 15x + 421, \\
x^5 + 35x^3 - 15x^2 + 185x - 1102,& x^5 - 85x^2 - 85x - 51,  & x^5 + 15x^3 - 45x^2 + 60x - 239, \\
 x^5 + 25x^3 - 125x^2 + 250x - 420, & x^5 - 50x^2 - 25x - 230, & x^5 + 2125x - 8075.
\end{eqnarray*}

To know which one corresponds to our elliptic curve, we just compute
the characteristic polynomial of the Frobenius at $3$, which
is given by $x^2-3$, which means that it has order $2$ in
$\PGL_2(\F_5)$. If we compute the inertial degree of $3$ in the above
extensions, we see that there exists a prime above $3$ with inertial
degree greater than $2$ in all the field extensions but $x^5-85x-153$,
which must be the extension we are looking for.

To search for the $2$-dimensional space $\coho^1(G_S,\Ad)$, we search
for a degree $20$ extension $M$ of $\Q$ fixed by the $C_6$ subgroup
(using the Pari script subfieldgen written by Bill Allomber). It is
given by the polynomial
\begin{multline*}
P(x)= x^{20} - 5 x^{19} + 5 x^{18} + 5 x^{17} + 105 x^{16} - 591 x^{15} +
  1545 x^{14} - 1125 x^{13} - 5975 x^{12} + 28195 x^{11}- \\  - 57199
  x^{10} + 44405 x^9 + 188910 x^8 - 778890 x^7 + 1946100 x^6 - 3335796
  x^5 +\\ + 4553305 x^4 -  4695185 x^3 + 3627665 x^2 - 1817365 x + 443586
\end{multline*}

Its degree $10$ subextension $N$ is given by 
\[
x^{10} - 5 x^9 + 5 x^8 + 10
x^7 - 15 x^6 + 40 x^5 - 155 x^4 + 350 x^3 - 430 x^2 + 1525 x - 2670.
\]

\begin{lemma}
  In the previous hypothesis, $\dim \coho^2(G_S,\Ad)=0$ and $\dim
  \coho^1(G_S,\Ad)=2$.
\end{lemma}

\begin{proof}
  By Lemma 9 in \cite{ravi2} $\dim \coho^2(G_S,\Ad)=r+s$ and $\dim
  \coho^1(G_S,\Ad)=r+s+2$, where $r = \dim \sha_S^1((\Ad)^*)$ and $s$ is 
  the number of primes $v \in S$ such that $\dim \coho^2(G_s,\Ad)\neq 0$.

  Since $17 \not\equiv \pm 1 \pmod 5$, Proposition 2 of \cite{ravi1} implies
  $\coho^2(G_{17},\Ad)=0$. Also,
  since the prime $5$ is totally ramified in the extension
  $\Q(\Ad)^{S_3 \times C_2}$, $\bar{\rho}|_{G_5}$ is indecomposable
  (it is not abelian), and also $\coho^2(G_5,\Ad)=0$ (the numbers in 
  Table 3 on \cite{ravi2} apply), so $s=0$. 

  On the other hand, the elements of $\sha_S^1((\Ad)^*)$ correspond to
  extensions of $\Q((\Ad)^*)$ unramified outside $S$ at which the
  primes above $5$ and $17$ split completely. In particular, they are
  unramified extensions of $M$. Since the class group of such
  extension is not divisible by $5$, we deduce that it is trivial, so
  $r=0$, and the result follows.
\end{proof}

\begin{remark}
  The local $\coho^1(G_5,\Ad)$ has dimension $3$, and the subspace
  $N_5$ is that of finite flat group schemes, which are
  indecomposable (see Remark\ref{rem:star}), which has dimension $1$
  (see Table 3 on \cite{ravi2}).
\end{remark}

We have to compute all degree $5$ Galois extensions of $M$ which are
unramified outside $5$ and $17$. We use Class Field Theory, where a bound for
the exponent of the modulus $e(\id{p})$ is given by the following
result.

\begin{prop}
Let $L/K$ be an abelian extension of prime degree $p$. If $\id{p}$
ramifies in $L/K$, then

\[
\left \{\begin{array}{cl}
e(\id{p}) = 1 & \text{if } \id{p} \nmid p\\
2
\le e(\id{p}) \le \left \lfloor\frac{p e(\id{p}|p)}{p-1} \right \rfloor +1 & \text{if }\id{p} \mid
 p.
\end{array} \right.
\]
\label{cohen}
\end{prop}

\begin{proof} See \cite{cohen} Proposition $3.3.21$ and Proposition $3.3.22$.
\end{proof}

We need a degree $5$ extension, and since the primes $5$ and $17$
ramify completely in $L_2$, the modulus is $\id{p}_5^{26}\id{p}_{17}$. We
compute such class group using Pari/GP (\cite{PARI}), and get that
such class group is isomorphic to
\[
C_{240}\times C_{40} \times C_{5} \times C_{5} \times C_{5} \times C_{5} \times C_{5}\times C_{5} \times C_{5} \times C_{5} \times C_{5} \times C_{5}.
\]

It should be pointed out that one can chose a basis of the characters
such that only one of them ramifies at $\id{p}_{17}$.

Recall that the extensions we are looking for come from degree $5$
extensions of $N$, but are not Galois over it. If we apply CFT to
$N$, the class group is isomorphic to
\[
C_{80} \times C_{20} \times C_5\times C_5 \times C_5,
\]
and no extension is ramified over the prime $17$.

\begin{lemma}
  If a rational prime $p$ is unramified in $\Q(\Ad)^{S_3\times C_2}$
  and has a prime over it with inertial degree $5$, then all primes
  dividing it have inertial degree $5$ and split completely in
  $L_2/\Q(\Ad)^{S_3\times C_2}$.
\label{lemma:removingbaselements}
\end{lemma}

\begin{proof}
  The maximal cyclic subgroup of $S_5$ of order divisible by $5$ is of
  order $5$, hence for any prime in $\Q(\Ad)$ the decomposition group
  is cyclic of order $5$. Since it does not intersect $S_3 \times
  C_2$, the first assertion follows and any prime over it in
  $\Q(\Ad)^{S_3\times C_2}$ splits completely in
  $\Q(\Ad)/\Q(\Ad)^{S_3\times C_2}$. But a cyclic group cannot be
  written as a semidirect product of groups whose order is divisible by
  $5$, hence it must split completely in $L/\Q(\Ad)$ as well.
\end{proof}

We search for $\coho^1(G_S,\Ad)$ computing all the elements in the
class group of $\Q(\Ad)^{C_6}$ that split completely for primes with
inertial degree $5$ in $\Q(\Ad)^{S_3\times C_2}/\Q$. With the first
such primes, we get a degree $5$ subspace, which contains a degree $3$
subspace coming from the class group of $\Q(\Ad)^{S_3 \times C_2}$, so
we get the $2$-dimensional subspace corresponding to the
$\coho^2(G_S,\Ad)$. Also, an important fact is that all the characters in this
$5$-dimensional space are unramified at $17$.

\begin{remark}
  To determine the number of auxiliary primes we need to add, we have
  to determine which elements in $\coho^1(G_S,\Ad)$ are trivial while
  restricting to $G_5$. Since the flat subspace is one dimensional
  (and the representations coming from our elliptic curve is in
  there), we are just led to prove which extensions give the same
  field extension of $\Q_5$. Note that since $\Q(\Ad)^{C_6}$ is
  totally ramified, and there are no solvable subgroups of $S_5$ whose
  order is divisible by $20$ and have order greater than $20$, the
  prime $\id{p}$ splits completely in $\Q(\Ad)/\Q(\Ad)^{C_6}$. Then if
  we restrict our representation to $G_5$, the representation we get
  has degree $20$ and is that of the completion of $M$ at
  $\id{p}_5$. 

  We can check the local behavior of our representations just by
  looking at the $5$-adic part of our character, and since our
  characters are only ramified at $5$, if two linearly independent
  ones have the same $5$-adic component, then the quotient would give
  a non-trivial unramified character of order $5$, but there are no
  such characters. Then one goes to zero (in $\coho^1(G_5,\Ad)/N_5$)
  and the other does not. In particular just one extra prime is
  enough.
\label{rem:star}
\end{remark}

We search for a prime $q \not\equiv \pm 1 \pmod 5$ and such that $a_q
\equiv \pm (q+1) \pmod {25}$, and $q=113$ is such a prime, since
$a_{113}=-14 \equiv -(113+1) \pmod{25}$.

\begin{lemma}
  There exists a weight $2$ modular form of level $17\cdot 113$ which
  is congruent modulo $5^2$ to the modular form attached to
  $E_{17a1}$.
\end{lemma}

\begin{proof}
  In view of the previous discussion, we just need to check that $113$
  is the right choice for the map
\[
\coho^1(G_{\{5,113\}},\Ad) \mapsto \coho^1(G_5,\Ad)/N_5 \times \coho^1(G_{113},\Ad)/N_{113},
\]
to be an isomorphism. We already know that the space
$\coho^1(G_{\{5\}},\Ad)$ is two dimensional, and that its image in
$\coho^1(G_5,\Ad)/N_5$ has dimension $1$ (the deformation $f_E$
corresponding to our elliptic curve maps to $0$), so we need to check
that the extra element is linearly independent with the non-zero
element in such cohomological group and that $f_E|_{G_{113}} \not \in
N_{113}$, which is equivalent to say that the Frobenius element at
$113$ modulo $5$ and modulo $25$ have different orders. Since
$a_{113}=-14$, the characteristic polynomial is given by $x^2+14x+113
\equiv (x-12)(x-24) \pmod{25}$, so the Frobenius element has order $4$
modulo $5$ and order $20$ modulo $25$.

We do the same computation as before, but adding this extra prime to
the ramification, and check that the cohomological dimension of
$\coho^1(G_{S \cup \{113\}},\Ad)$ increases by 1 (the whole $\F_5$
vector space computed using CFT has dimension $17$, but using
Lemma~\ref{lemma:removingbaselements} we get a $7$-dimensional
subspace, and the ones coming from $N$
satisfying the same property have dimension $4$). We just need to
check that the $5$-adic characters corresponding to these
$3$-dimensional subspace generate a $3$-dimensional space. Note that
we can chose a basis such that there is a $2$-dimensional part
$\<v_1,v_2>$ unramified at $113$, and a one dimensional part $\<v_3>$
ramified also at $113$. If the $5$-adic character of $v_3$ is in the
vector space spanned by $\<v_1,v_2>$, then we can multiply $v_3$ by
the inverse of the $5$-adic part of the character (which exists
globally) to get an extension in our subspace (which does not come from
$N$) only ramified at $113$. But using CFT, it
is easy to check that the only subspace satisfying
Lemma~\ref{lemma:removingbaselements} comes from an abelian extension
of $N$.
\end{proof}

\begin{remark}
  In this particular case, one can search for the form in the right
  space. We did such computation using Magma (\cite{magma}) and
  computed the space of newforms of level $17\cdot 113$. Such space
  contains (up to conjugation) $5$ newforms. Our curve is isomorphic
  modulo $25$ to an eigenform whose coefficient field has degree $43$
  over $\Q$, given by the polynomial
{\footnotesize  
\begin{multline*}
x^{43} - 6x^{42} - 52x^{41} + 373x^{40}+ 1125x^{39} - 10604x^{38} -
  11821x^{37} + 182630x^{36} + 26405x^{35} - 2127738x^{34} +
  979653x^{33} \\+ 17730287x^{32} - 15815881x^{31} - 108925194x^{30} +
  134740636x^{29} + 500970519x^{28} - 774455464x^{27} -
  1732542039x^{26} \\+ 3221093358x^{25} + 4479749953x^{24} -
  9965892052x^{23} - 8501952587x^{22} + 23170021972x^{21} +
  11368626528x^{20} \\- 40486609059x^{19} - 9675455698x^{18} +
  52796933022x^{17} + 3349112852x^{16} - 50684587408x^{15} +
  2843708080x^{14} \\+ 35061372555x^{13} - 4639214583x^{12} -
  16918972986x^{11} + 2949253955x^{10} + 5411942205x^9 - 1031364938x^8\\
  - 1053178460x^7 + 201802209x^6 + 106332326x^5 - 20249486x^4 -
  3919101x^3 + 714966x^2 + 1842x - 263.
\end{multline*}}%
\end{remark}

\bibliographystyle{alpha}
\bibliography{biblio}

\newcommand{\etalchar}[1]{$^{#1}$}
\begin{thebibliography}{DFG04}

\bibitem[BCP97]{magma}
Wieb Bosma, John Cannon, and Catherine Playoust.
\newblock The {M}agma algebra system. {I}. {T}he user language.
\newblock {\em J. Symbolic Comput.}, 24(3-4):235--265, 1997.
\newblock Computational algebra and number theory (London, 1993).

\bibitem[BD]{BD}
Christophe Breuil and Fred Diamond.
\newblock Formes modulaires de hilbert modulo p et valeurs d'extensions entre
  caractères galoisiens.
\newblock {\em Ann. Scient. de l'E.N.S., to appear}.

\bibitem[Car89]{carayol}
Henri Carayol.
\newblock Sur les repr\'esentations galoisiennes modulo {$l$} attach\'ees aux
  formes modulaires.
\newblock {\em Duke Math. J.}, 59(3):785--801, 1989.

\bibitem[Coh00]{cohen}
Henri Cohen.
\newblock {\em Advanced topics in computational number theory}, volume 193 of
  {\em Graduate Texts in Mathematics}.
\newblock Springer-Verlag, New York, 2000.

\bibitem[DFG04]{diamond3}
Fred Diamond, Matthias Flach, and Li~Guo.
\newblock The {T}amagawa number conjecture of adjoint motives of modular forms.
\newblock {\em Ann. Sci. \'Ecole Norm. Sup. (4)}, 37(5):663--727, 2004.

\bibitem[DS05]{Diamond}
Fred Diamond and Jerry Shurman.
\newblock {\em A First Course in Modular Forms}.
\newblock Springer, 2005.

\bibitem[Dum05]{dummigan}
Neil Dummigan.
\newblock Level-lowering for higher congruences of modular forms.
\newblock 2005.
\newblock \url{ http://www.neil-dummigan.staff.shef.ac.uk/levell4.dvi}.

\bibitem[Ger12]{gera}
David Geraghty.
\newblock Modularity lifting theorems for ordinary representations.
\newblock 2012.
\newblock Available at \url{http://www.math.ias.edu/~geraghty/}.

\bibitem[JR13]{jones}
John~W. Jones and David~P. Roberts.
\newblock A database of number fields.
\newblock \url{http://hobbes.la.asu.edu/NFDB}, 2013.

\bibitem[PAR13]{PARI}
{\em {PARI/GP, version {\tt 2.5.5}}}.
\newblock Bordeaux, 2013.
\newblock available from \url{http://pari.math.u-bordeaux.fr/}.

\bibitem[Ram93]{ravi1}
Ravi Ramakrishna.
\newblock On a variation of {M}azur's deformation functor.
\newblock {\em Compositio Math.}, 87(3):269--286, 1993.

\bibitem[Ram99]{ravi2}
Ravi Ramakrishna.
\newblock Lifting {G}alois representations.
\newblock {\em Invent. Math.}, 138(3):537--562, 1999.

\bibitem[Ram02]{ravi3}
Ravi Ramakrishna.
\newblock Deforming {G}alois representations and the conjectures of {S}erre and
  {F}ontaine-{M}azur.
\newblock {\em Ann. of Math. (2)}, 156(1):115--154, 2002.

\bibitem[RH08]{ravihamblen}
Ravi Ramakrishna and Spencer Hamblen.
\newblock Deformation of certain reducible galois representations, ii.
\newblock {\em American Journal of Mathematics}, 130, 2008.

\bibitem[Rib85]{Ribet}
Kenneth~A. Ribet.
\newblock On {$l$}-adic representations attached to modular forms. {II}.
\newblock {\em Glasgow Math. J.}, 27:185--194, 1985.

\bibitem[S{\etalchar{+}}13]{sage}
W.\thinspace{}A. Stein et~al.
\newblock {\em {S}age {M}athematics {S}oftware ({V}ersion 5.8)}.
\newblock The Sage Development Team, 2013.
\newblock {\tt http://www.sagemath.org}.

\bibitem[Ser89]{serreladic}
Jean-Pierre Serre.
\newblock {\em Abelian $\ell$-adic representations and elliptic curves}.
\newblock The advanced book program. Addison-Wesley publishing company, 1989.

\end{thebibliography}

\end{document}